\documentclass[12pt]{amsart}
\usepackage{amsmath,amssymb,graphicx,amsthm,tikz,color,tabu,amscd,mathrsfs}
\usepackage{hyperref}
\usepackage{fullpage}
\usetikzlibrary{shapes,arrows}
\newtheorem{theorem}{\textbf{Theorem}}[section]
\newtheorem{proposition}[theorem]{\textbf{Proposition}}
\newtheorem{lemma}[theorem]{\textbf{Lemma}}
\theoremstyle{definition}
\newtheorem{definition}[theorem]{\textbf{Definition}}
\newtheorem{remark}[theorem]{\textbf{Remark}}
\newtheorem{example}[theorem]{Example}
\numberwithin{equation}{section}

\begin{document}
\title{A Yang-Baxter Equation for Metaplectic Ice}
\author{Ben Brubaker}
\address{School of Mathematics, University of Minnesota, Minneapolis, MN 55455}
\email{brubaker@math.umn.edu}
\author{Valentin Buciumas}
\address{Einstein Institute of Mathematics, Edmond J. Safra Campus, Givat Ram, The Hebrew
University of Jerusalem, Jerusalem, 91904, Israel}
\email{valentin.buciumas@gmail.com}
\author{Daniel Bump}
\address{Department of Mathematics, Stanford University, Stanford, CA 94305-2125}
\email{bump@math.stanford.edu}
\maketitle

    
\newcommand{\gammaice}[6]{
\resizebox{1.75cm}{1.75cm}{\begin{tikzpicture}
\coordinate (a) at (-.75, 0);
\coordinate (b) at (0, .75);
\coordinate (c) at (.75, 0);
\coordinate (d) at (0, -.75);
\coordinate (aa) at (-.75,.5);
\coordinate (cc) at (.75,.5);
\draw (a)--(c);
\draw (b)--(d);
\draw[fill=white] (a) circle (.25);
\draw[fill=white] (b) circle (.25);
\draw[fill=white] (c) circle (.25);
\draw[fill=white] (d) circle (.25);
\node at (0,1) { };
\node at (a) {$#1$};
\node at (b) {$#2$};
\node at (c) {$#3$};
\node at (d) {$#4$};
\node at (aa) {$#5$};
\node at (cc) {$#6$};
\end{tikzpicture}}
}
\newcommand{\gamgam}[8]{
\resizebox{1.75cm}{1.95cm}{\begin{tikzpicture}
\coordinate (a) at (-.75, -.75);
\coordinate (b) at (-.75, .75);
\coordinate (c) at (.75, .75);
\coordinate (d) at (.75, -.75);
\coordinate (aa) at (-.75, -1.25);
\coordinate (bb) at (-.75, 1.25);
\coordinate (cc) at (.75, 1.25);
\coordinate (dd) at (.75, -1.25);
\draw (a)--(c);
\draw (b)--(d);
\draw[fill=white] (a) circle (.25);
\draw[fill=white] (b) circle (.25);
\draw[fill=white] (c) circle (.25);
\draw[fill=white] (d) circle (.25);
\node at (0,1) { };
\node at (a) {$#1$};
\node at (b) {$#2$};
\node at (c) {$#3$};
\node at (d) {$#4$};
\node at (aa) {$#5$};
\node at (bb) {$#6$};
\node at (cc) {$#7$};
\node at (dd) {$#8$};
\end{tikzpicture}}
}
\newcommand{\gamgamsmall}[8]{
\resizebox{1.25cm}{1.95cm}{\begin{tikzpicture}
\coordinate (a) at (-.75, -.75);
\coordinate (b) at (-.75, .75);
\coordinate (c) at (.75, .75);
\coordinate (d) at (.75, -.75);
\coordinate (aa) at (-.75, -1.25);
\coordinate (bb) at (-.75, 1.25);
\coordinate (cc) at (.75, 1.25);
\coordinate (dd) at (.75, -1.25);
\draw (a)--(c);
\draw (b)--(d);
\draw[fill=white] (a) circle (.25);
\draw[fill=white] (b) circle (.25);
\draw[fill=white] (c) circle (.25);
\draw[fill=white] (d) circle (.25);
\node at (0,1) { };
\node at (a) {$#1$};
\node at (b) {$#2$};
\node at (c) {$#3$};
\node at (d) {$#4$};
\node at (aa) {$#5$};
\node at (bb) {$#6$};
\node at (cc) {$#7$};
\node at (dd) {$#8$};
\end{tikzpicture}}
}

%
\newcommand{\lhs}[9]{\raisebox{-30pt}{\resizebox{5.4cm}{4.5cm}{\begin{tikzpicture}
\path[use as bounding box](-1,-2) rectangle(5,3);
\coordinate (a) at (0,-1);     
\coordinate (b) at (0,-1.5);   
\coordinate (c) at (0,1);      
\coordinate (d) at (0,1.5);    %
\coordinate (e) at (3,2);      
\coordinate (f) at (4,1);      
\coordinate (g) at (4,1.5);    %
\coordinate (h) at (4,-1);     
\coordinate (i) at (4,-1.5);   %
\coordinate (j) at (3,-2);     
\coordinate (k) at (2,1);      
\coordinate (l) at (2,1.5);    %
\coordinate (m) at (2,-1);     
\coordinate (n) at (2,-1.5);   %
\coordinate (o) at (3,0);      
\coordinate (p) at (1,0);      
\coordinate (q) at (3,1);      %
\coordinate (r) at (3,-1);     %
\draw (a) to [out=0, in=180] (k);
\draw (k)--(f);
\draw (c) to [out=0, in=180] (m);
\draw (m)--(h);
\draw (e)--(j);
\draw[fill=white] (a) circle (.25);
\draw[fill=white] (c) circle (.25);
\draw[fill=white] (e) circle (.25);
\draw[fill=white] (f) circle (.25);
\draw[fill=white] (h) circle (.25);
\draw[fill=white] (j) circle (.25);
\draw[fill=white] (k) circle (.25);
\draw[fill=white] (m) circle (.25);
\draw[fill=white] (o) circle (.25);
\path[fill=white] (p) circle (.25);
\path[fill=white] (q) circle (.35);
\path[fill=white] (r) circle (.35);
\node at (a) {$#1$};
\node at (c) {$#2$};
\node at (e) {$#3$};
\node at (f) {$#4$};
\node at (h) {$#5$};
\node at (j) {$#6$};
\node at (k) {$#7$};
\node at (m) {$#8$};
\node at (o) {$#9$};
\node at (p) {$\scriptstyle R_{z_i,z_j}$};
\node at (q) {$z_j$};
\node at (r) {$z_i$};
\end{tikzpicture}}}}
\newcommand{\rhs}[9]{\raisebox{-30pt}{\resizebox{5.4cm}{4.5cm}
{\begin{tikzpicture}
\path[use as bounding box](-1,-2) rectangle(5,3);
\coordinate (a) at (0,-1); 
\coordinate (b) at (0,1);  
\coordinate (c) at (1,2);  
\coordinate (d) at (4,1);  
\coordinate (e) at (4,-1); 
\coordinate (f) at (1,-2); 
\coordinate (g) at (2,-1); 
\coordinate (h) at (2,1);  
\coordinate (i) at (1,0);  
\coordinate (r) at (3,0);  
\coordinate (s) at (1,-1);  
\coordinate (t) at (1,1);  
\draw (a)--(g) to [out=0, in=180] (d);
\draw (b)--(h) to [out=0, in=180] (e);
\draw (c)--(f);
\draw[fill=white] (a) circle (.25);
\draw[fill=white] (b) circle (.25);
\draw[fill=white] (c) circle (.25);
\draw[fill=white] (d) circle (.25);
\draw[fill=white] (e) circle (.25);
\draw[fill=white] (f) circle (.25);
\draw[fill=white] (g) circle (.25);
\draw[fill=white] (h) circle (.25);
\draw[fill=white] (i) circle (.25);
\path[fill=white] (r) circle (.25);
\path[fill=white] (s) circle (.35);
\path[fill=white] (t) circle (.35);
\node at (a) {$#1$};
\node at (b) {$#2$};
\node at (c) {$#3$};
\node at (d) {$#4$};
\node at (e) {$#5$};
\node at (f) {$#6$};
\node at (g) {$#7$};
\node at (h) {$#8$};
\node at (i) {$#9$};
\node at (r) {$\scriptstyle R_{z_i,z_j}$};
\node at (s) {$z_j$};
\node at (t) {$z_i$};
\end{tikzpicture}}}}

\newcommand{\botcharge}[2]{\raisebox{-32pt}{$\begin{array}{cc}#1\\\scriptstyle #2\end{array}$}}
\newcommand{\topcharge}[2]{\raisebox{25pt}{$\begin{array}{cc}\scriptstyle #2\\#1\end{array}$}}
\newcommand{\GL}{\operatorname{GL}}
\newcommand{\SL}{\operatorname{SL}}
\newcommand{\C}{\mathbb{C}}


\begin{abstract}
  We will give new applications of quantum groups to the study
  of spherical Whittaker functions on the metaplectic $n$-fold
  cover of $\GL(r,F)$, where $F$ is a nonarchimedean local field.
  Earlier Brubaker, Bump, Friedberg, Chinta and Gunnells had shown
  that these Whittaker functions can be identified with
  the partition functions of statistical mechanical systems.
  They postulated that a Yang-Baxter equation underlies
  the properties of these Whittaker functions. We confirm
  this, and identify the corresponding Yang-Baxter equation
  with that of the quantum affine Lie superalgebra
  $U_{\sqrt{v}}(\widehat{\mathfrak{gl}}(1|n))$, modified by Drinfeld
  twisting to introduce Gauss sums. (The deformation parameter $v$
  is specialized to the inverse of the residue field cardinality.)

  For principal series representations of metaplectic groups, the Whittaker models are not unique. 
  The scattering matrix for the standard intertwining
  operators is vector valued. For a simple reflection, it was
  computed by Kazhdan and Patterson, who applied it to
  generalized theta series. We will show that the scattering
  matrix on the space of Whittaker functions for a simple
  reflection coincides with the twisted $R$-matrix of the quantum group
  $U_{\sqrt{v}}(\widehat{\mathfrak{gl}}(n))$. This is a piece
  of the twisted $R$-matrix for $U_{\sqrt{v}}(\widehat{\mathfrak{gl}}(1|n))$,
  mentioned above.
\end{abstract}

\section{Introduction}
The formula of Casselman and Shalika~\cite{CasselmanShalika}
expresses values of the spherical Whittaker function for a principal series
representation of a reductive algebraic group over a $p$-adic field in terms
of the characters of irreducible finite-dimensional representations of the
Langlands dual group. Their proof relies on knowing the effect
of the intertwining integrals on the normalized Whittaker functional. Since
the Whittaker functional is unique, the intertwining integral just multiplies
it by a constant, which they computed.

In contrast with this algebraic case, Whittaker models of principal series
representations of metaplectic groups are generally not unique. The
effect of the intertwining operators on the Whittaker models was computed by
Kazhdan and Patterson~\cite{KazhdanPatterson}. Specifically, they computed
the scattering matrix of the intertwining operator corresponding to a simple
reflection on the finite-dimensional vector space of Whittaker functionals
for the $n$-fold metaplectic cover of $\GL(r,F)$, where $F$ is a $p$-adic
field. Some terms in this matrix are simple rational functions of the
Langlands parameters, while others involve $n$-th order Gauss sums. Though
complicated in appearance, this scattering matrix was a key ingredient in their
study of generalized theta series, and also in the later development of a
metaplectic Casselman-Shalika formula by Chinta and Offen~\cite{ChintaOffen}
and McNamara~\cite{McNamaraCS}.

One of the two main results of this paper is that this scattering matrix computed by Kazhdan
and Patterson is the $R$-matrix of a quantum group, quantum affine
$\mathfrak{gl}(n)$, modified by Drinfeld twisting to introduce Gauss sums. 
This appears to be a new connection between the representation theory
of $p$-adic groups and quantum groups, which should allow one to use
techniques from the theory of quantum groups to study metaplectic Whittaker functions.

Although we can now prove this directly, we were led to this result by
studying lattice models whose partition functions give values of Whittaker
functions on a metaplectic cover of $\GL(r,F)$. In \cite{mice}, it was predicted
that a solvable such model should exist; i.e., one for which a solution to the Yang-Baxter equation
exists. Such a solvable model has important applications in number theory: it
gives easy proofs (in the style of Kuperberg's proof of the alternating sign
matrix conjecture) of several facts about Weyl group multiple Dirichlet series
\cite{wmd5book}. The other main result of this paper is the discovery of a
{\it solvable} lattice model whose partition function is a metaplectic Whittaker function. 
Moreover, we relate this solution to an $R$-matrix for the quantum affine superalgebra
$\mathfrak{gl}(1|n)$. The relation between the two main results follows from
the inclusion of (quantum affine) $\mathfrak{gl}(n)$ into
$\mathfrak{gl}(1|n)$.

We now explain these results in more detail. Let $\widetilde{G}$ denote an
$n$-fold metaplectic cover of $G := \GL (r, F)$ where the non-archimedean
local field $F$ contains the $2n$-th roots of unity. Given a partition $\lambda$
of length $\leqslant r$, we will exhibit a system $\mathfrak{S}_\lambda$ 
whose partition function equals the value of one
particular spherical Whittaker function at 
$\mathbf{s}\bigl(\operatorname{diag}(p^{\lambda_1}, \ldots, p^{\lambda_r})\bigr)$,
where $\mathbf{s}:\GL(r,F)\rightarrow\widetilde{G}$ is a standard section.

The systems proposed in~\cite{mice} were generalizations of the six-vertex
model. The six-vertex model with field-free boundary conditions was solved by
Lieb~\cite{Lieb}, Sutherland~\cite{Sutherland} and Baxter~\cite{Baxter} and
were motivating examples that led to the discovery of quantum groups
(cf.~{\cite{KulishReshetikhinSklyanin, JimboQuantumGroups, Drinfeld}}). In
Baxter's work, the solvability of the models is dictated by the Yang-Baxter
equation where the relevant quantum group is $U_q(\widehat{\mathfrak{sl}}_2)$. In
the special case $n=1$ (so when we are working with non-metaplectic $\operatorname{GL}(r,F)$), the systems proposed in~\cite{mice} 
are six-vertex models that coincide with those discussed in Brubaker, Bump and Friedberg~\cite{hkice,wmd5book} and
there is a Yang-Baxter equation available. However even in this case these
models differ from those considered by Lieb, Sutherland and Baxter since
they are not field-free. Based on
the results of this paper, we now understand that the relevant quantum group for the lattice models in
\cite{hkice,wmd5book} is
$U_q\bigl(\widehat{\mathfrak{gl}}(1|1)\bigr)$, as we will make clear in
subsequent sections.

It was explained in~{\cite{mice}} that a Yang-Baxter equation for metaplectic
ice would give new proofs of two important results in the theory of
metaplectic Whittaker functions. The first is a set of local functional
equations corresponding to the permutation of the Langlands-Satake
parameters. The second is an equivalence of two explicit formulas for the
Whittaker function, leading to analytic continuation and functional equations
for associated Weyl group multiple Dirichlet series. The proof of this latter
statement occupies the majority of~{\cite{wmd5book}}. 

However, no Yang-Baxter equation for the metaplectic ice in~{\cite{mice}}
could be found. In this paper we will make a small but crucial modification
of the Boltzmann weights for the model in~{\cite{mice}}. This change does not
affect the partition function, but it makes possible a Yang-Baxter
equation. This is Theorem~\ref{gammaybe} in Section~\ref{ybesection}. The
solutions to the Yang-Baxter equation may be encoded in a matrix commonly
referred to as an \textit{$R$-matrix.}

We further prove that the resulting $R$-matrix has two important properties:
\begin{enumerate}
\item It is a Drinfeld twist of the
$R$-matrix obtained from the defining representation of quantum affine $\widehat{\mathfrak{gl}} (1| n)$, a Lie
superalgebra.
\medskip
\item It contains the $R$-matrix of a Drinfeld twist of $\widehat{\mathfrak{gl}}(n)$
  which, as we have already explained, we will identify with the
  scattering matrix of intertwining operators on Whittaker models for metaplectic principal series.
\end{enumerate}

Consider the quantized enveloping algebra of the untwisted affine Lie algebra
$\widehat{\mathfrak{g}\mathfrak{l}} (n)$, i.e. the central extension of the loop algebra of $\mathfrak{gl}(n)$. We denote the quantized enveloping
algebra as $U_{\sqrt{v}} (\widehat{\mathfrak{g}\mathfrak{l}} (n))$ instead of the
usual $U_q$ because in our application the deformation parameter $v$ will be
$q^{-1}$, where $q$ is the cardinality of the residue field of $F$.  If $V$
and $W$ are vector spaces, let $\tau=\tau_{V,W}$ denote the flip operator
$V\otimes W\to W\otimes V$.  The Hopf algebra $U_{\sqrt{v}}
(\widehat{\mathfrak{g}\mathfrak{l}} (n))$ is almost quasitriangular; given any two modules $V$ and $W$, there is an {\textit{$R$-matrix}}
$R_{V, W} \in \operatorname{End} (V \otimes W)$ such that $\tau R_{V, W} : V
\otimes W \longrightarrow W \otimes V$ is a module homomorphism (though it will not always be an isomorphism). 
The
$R$-matrices for $U_{\sqrt{v}} ( \widehat{\mathfrak{g}\mathfrak{l}} (n))$ acting on
a tensor product of two evaluation modules were found by Jimbo~{\cite{JimboToda}} (see also Frenkel
and Reshetikhin~{\cite{FrenkelReshetikhin}}, Remark~4.1.); they satisfy a parametrized Yang-Baxter equation. 

The quantum group $U_{\sqrt{v}} ( \widehat{\mathfrak{g}\mathfrak{l}} (n))$ 
has an $n$-dimensional evaluation module $V_+ (z)$ for every complex parameter value
$z$. We will label a basis of the module $v_{+ a} (z)$
where $a$ runs through the integers modulo $n$. The parameter $+ a$ will be
called a {\textit{positive decorated spin}} (to be supplemented later by another one, 
denoted $- 0$).
We may think of the decoration $a$ (mod $n$) as roughly
corresponding to the sheets of the metaplectic cover $\widetilde{G}
\longrightarrow \GL (r)$ of degree~$n$. 

The resulting $R$-matrix in $\operatorname{End} (V_+(z_1) \otimes V_+(z_2))$ is
the matrix $R_{z_1, z_2} := R_{V_+(z_1), V_+(z_2)}$ whose entries $R_{\alpha, \beta}^{\gamma, \delta} (z_1, z_2)$ are
indexed by positive decorated spins $\alpha, \beta, \gamma$ and $\delta$ such that 
\[R_{z_1, z_2}\bigl(v_{\alpha}(z_1) \otimes v_{\beta}(z_2)\bigr) 
= \sum_{\gamma, \delta} R_{\alpha, \beta}^{\gamma, \delta} (z_1, z_2) 
v_{\gamma}(z_1) \otimes v_{\beta}(z_2).\]
These values are given by the following table:
\[ {\tabulinesep=1.2mm
    \begin{tabu}{|c|c|c|c|}
     \hline
     \alpha, \beta, \gamma, \delta & \begin{array}{c}
       + a, + a, + a, + a\\
       (0 \leqslant a \leqslant n)
     \end{array} & \begin{array}{c}
       + b, + a, + b, + a\\
       (0 \leqslant a, b \leqslant n, a \neq b)
     \end{array} & \begin{array}{c}
       + b, + a, + a, + b\\
       (0 \leqslant a, b \leqslant n, a \neq b)
     \end{array}\\
      \hline
      R_{\alpha, \beta}^{\gamma, \delta} (z_1, z_2) & \frac{- v + (z_1 /
        z_2)^{n}}{1 - v (z_1 / z_2)^{n}} & g (a - b)
      \frac{1 - (z_1 / z_2)^{ n}}{1 - v (z_1 / z_2)^{n}} & 
      \left\{\begin{array}{ll}
      (1 - v) \frac{(z_1 / z_2)^{n}}{1 - v (z_1 / z_2)^{n}} &
      a>b,\\\\
      (1 - v) \frac{1}{1 - v (z_1 / z_2)^{n}} &
      a<b.
      \end{array}\right.
      \\
      \hline
   \end{tabu}} \]
Here $g (a - b)$ is an $n$-th order Gauss sum. These are not present in the
out-of-the-box $U_{\sqrt{v}} ( \widehat{\mathfrak{g}\mathfrak{l}} (n))$ $R$-matrix, but
may be introduced by Drinfeld twisting that will be discussed in
Section~\ref{supersection} (see also Section~4 of~\cite{BBBF}). This procedure
does not affect the validity of the Yang-Baxter equations, but is needed for
comparison with the $R$-matrix for the partition functions of metaplectic ice
giving rise to Whittaker functions.

To obtain the full $R$-matrix used in the Yang-Baxter equation for metaplectic
ice, we must enlarge the set of positive decorated spins $+ a$ to include one more,
labelled $-0$. Thus there are $n+1$ decorated spins altogether, the positive ones
and one more. The $n$-dimensional vector space $V_+ (z)$ is
enlarged to an $n + 1$ ``super'' vector space $V_{\pm} (z)$. The
positive decorated spins are a basis for the odd part $V_+(z)$,
and the even part $V_-(z)$ is one-dimensional, spanned by $-0$.
In Section \ref{ybesection}, we present an $R$-matrix that gives a solution of
the Yang-Baxter equation for the metaplectic ice model.  In
Section~\ref{supersection}, we show that the solution of the Yang-Baxter
equation is equivalent to the $R$-matrix corresponding to the defining
representation of the quantum affine Lie superalgebra $U_{\sqrt{v}}(
\widehat{\mathfrak{g}\mathfrak{l}} (1|n))$ modified by a Drinfeld twist.

Finally, we explain the connection between the $R$-matrix of Theorem~\ref{gammaybe}
and the structure constants alluded to in item (2) above. The local functional
equations for metaplectic Whittaker functions mentioned earlier may be understood 
as arising from intertwining operators. Let $\widehat{T}$ be the diagonal torus in $\GL
(r, \mathbb{C})$, the Langlands dual group of $G$. Each diagonal matrix
\[ \mathbf{z} = \left( \begin{array}{ccc}
     z_1 &  & \\
     & \ddots & \\
     &  & z_{r}
   \end{array} \right) \in \widehat{T} ( \mathbb{C}) \]
indexes a principal series representation $\pi_{\mathbf{z}}$ of $\widetilde{G}$.
Let $\mathcal{W}^{\mathbf{z}}$ be the finite-dimensional vector space of
spherical Whittaker functions for $\pi_{\mathbf{z}}$. If $n = 1$, $\mathcal{W}^{\mathbf{z}}$ is
one-dimensional, but not in general since if $n > 1$ the representation
$\pi_{\mathbf{z}}$ does not have unique Whittaker models. If $s_i$ is a simple
reflection in the Weyl group $W$, then let $\mathcal{A}_{s_i}$ denote the
standard \textit{intertwining integral} $\mathcal{A}_{s_i} :
\pi_{\mathbf{z}} \longrightarrow \pi_{s_i \mathbf{z}}$ (see \eqref{intertwinerunnorm} for
the precise definition). This induces a map
$\mathcal{W}^{\mathbf{z}} \to \mathcal{W}^{s_i \mathbf{z}}$. If $n > 1$ then
$\mathcal{A}_{s_i}$ has an interesting scattering matrix on the Whittaker
model that was computed by Kazhdan and Patterson
(Lemma~I.3.3 of ~{\cite{KazhdanPatterson}}). This calculation
underlies their work on generalized theta series, and was used by
Chinta and Offen~{\cite{ChintaOffen}} and generalized by
McNamara~{\cite{McNamaraCS}} to study the analog of the
Casselman-Shalika formula for the spherical Whittaker functions.

Let $\pi_{\mathbf{z}, \psi}$ be the module of Whittaker coinvariants of the
representation $\pi_{\mathbf{z}}$. By definition this is the quotient of the
underlying space of $\pi_{\mathbf{z}}$ characterized by the fact that a
linear functional is a Whittaker functional if and only if it factors through
$\pi_{\mathbf{z}, \psi}$. Thus $\pi_{\mathbf{z}, \psi}$ is the dual space
of the space of Whittaker functionals on $\pi_{\mathbf{z}}$. Its dimension
is $n^r$.
In Section~\ref{connex}, we will prove that the scattering matrix of the intertwining
integrals on the Whittaker coinvariants is essentially 
$\tau R_{z_i,z_{i+1}}$, where $R_{z_i,z_{i+1}}$ is the $R$-matrix
for a Drinfeld twist of $U_{\sqrt{v}}(\widehat{\mathfrak{gl}}(n))$.


\begin{theorem}
\label{commdiagwithintertwiner}
There is an isomorphism $\theta_{\mathbf{z}}$ of the space
$\pi_{\mathbf{z},\psi}$ of Whittaker coinvariants
to the vector space $V_+ (z_1) \otimes \cdots \otimes V_+ (z_r)$ that takes the vectors
$v_{+ a_1} (z_1) \otimes \cdots \otimes v_{+a_r} (z_r)$ into the basis of $\pi_{\mathbf{z},\psi}$
dual to the basis of $\mathcal{W}^{\mathbf{z}}$ given
in~\cite{KazhdanPatterson, ChintaOffen, McNamaraCS} (see Section~\ref{connex}). Then the following diagram commutes:
\[\begin{CD}\pi_{\mathbf{z},\psi}@>\theta_{\mathbf{z}}>>V_+(z_1) \otimes \cdots \otimes V_+(z_i)\otimes V_+(z_{i+1})\otimes\cdots\otimes V_+(z_r)\\
@VV\bar{\mathcal{A}}_{s_i} V @VV (\tau R_{z_i,z_{i+1}})_{i,i+1} V\\
\pi_{s_i\mathbf{z},\psi}@>\theta_{\mathbf{s_iz}}>>V_+(z_1) \otimes \cdots \otimes V_+(z_{i+1})\otimes V_+(z_i)\otimes\cdots\otimes V_+(z_r)
\end{CD} \]
where $\bar{\mathcal{A}}_{s_i}$ denotes the map induced by the normalized intertwining operator
defined in \eqref{normalizedint}.
\end{theorem}

The notation $(\tau R_{z_i,z_{i+1}})_{i,i+1}$ means that the operator
$\tau R_{z_i,z_{i+1}}:V_+(z_i)\otimes V_+(z_{i+1})\to V_+(z_{i+1})\otimes V_+(z_i)$
is applied to the $i,i+1$ tensor components, while we take the identity map on the
remaining components.

This offers a new and seemingly fundamental connection between the
representation theory of quantum groups and $p$-adic metaplectic groups.
It also suggests several immediate questions.

First, one may ask for generalizations 
to other Cartan types. For symplectic groups, Yang-Baxter equations based on those
found here are given in Gray~\cite{GrayThesis}. A categorical framework for some of
these operations would be desirable. Even for central extensions of $\GL(r,F)$ there 
are open questions. We required the $2n$-th roots of unity to be in the ground field $F$,
in order to twist the Matsumoto cocycle defining the metaplectic
central extension of $\GL(r,F)$ by a cocycle of the form
$(\det(g_1),\det(g_2))_{2n}$ as in~(\ref{toruscocycle}). We may ask whether
other choices of cocycle admit a similar story; in particular, some choices result in a strictly
smaller dimensional space of Whittaker models, so wouldn't biject with basis elements in the
tensor product of vector spaces appearing in Theorem~\ref{commdiagwithintertwiner}.

One may also ask for connections with other literature such as
Weissman~\cite{Weissman}. It seems particularly important to understand the
relation between our work and the the quantum geometric Langlands program
initiated by Lurie and Gaitsgory in \cite{GaitsgoryWhittaker}, and more specifically the
relation to the work of Lysenko~\cite{Lysenko} and Gaitsgory and Lysenko
\cite{GaitsgoryLysenko}.

\begin{remark}\label{missingmodule}
  While we have given an interpretation of the $n+1$
  decorated spins which are the possible states of the horizontal edges as basis
  vectors for an evaluation module of
  $U_{\sqrt{v}} (\widehat{\mathfrak{g}\mathfrak{l}} (1|n))$, the edges of
  vertical type have no known similar interpretation. One may
  ask whether $U_{\sqrt{v}} (\widehat{\mathfrak{g}\mathfrak{l}} (1|n))$
  has a two-dimensional module $M$ such that
  the Boltzmann weights in Figure~\ref{gammaice} are interpreted
  as the $R$-matrix for the pair $V_z$, $M$. We know no reason
  for such an $M$ to exist, except that if it does not, then
  Theorem~\ref{gammaybe} is an example of a parametrized
  Yang-Baxter equation that is not predicted by quasitriangularity.
\end{remark}

We conclude by reviewing some recent papers which are sequels to
this one.

The paper by Brubaker, Buciumas, Bump and Friedberg~\cite{BBBF} was written
after the first draft of this one was already posted to the arxiv, and
depends on this one. In it we give a very general method of constructing
representations of the affine Hecke algebra and show that examples of such
representations can come either from the theory of Whittaker functionals on
metaplectic $p$-adic groups or from certain Schur-Weyl dualities for quantum
affine algebras. Theorem~1 in the present paper is used to prove the two
representations mentioned are in fact the same.  The paper also contains a
more formal discussion of the Drinfeld twisting, an important supplement to
the brief treatment we give below in Section~\ref{twisting}.

The paper by Brubaker, Buciumas, Bump and Gray~\cite{BBBG} was also
written after this one. It uses the Yang-Baxter equations
from this paper, and supplementary ones from Gray~\cite{GrayThesis},
to reprove the main result of~\cite{wmd5book}, which may be
expressed as the equality of the partition functions of two different ice models.
One of the two ice models is described below in Section~\ref{partitionfunction}.
The other one is similar but has different weights. The equality of the two
partition functions is reminiscent of dualities for physical systems,
similar for example to the Kramers-Wannier duality that relates the partition
functions of the low-temperature and high temperature Ising models.

In the paper Brubaker, Buciumas, Bump and Gustafsson~\cite{BBBGu}
it is shown (extending the earlier paper \cite{BrubakerSchultz} in
the $n=1$ case) that the row transfer matrices for metaplectic ice
can be interpreted as operators on the Fermionic Fock space $\mathfrak{F}$
of Kashiwara, Miwa and Stern~\cite{KMS}, after Drinfeld twisting.
This is a module for (twisted) $U_{\sqrt{v}}(\widehat{\mathfrak{sl}}_n)$.
To achieve this one modifies the boundary conditions so that the
grid has infinitely many columns. Then a sequence of spins in
a row of vertically oriented edges may be interpreted as a basis
vector in $\mathfrak{F}$, and the main theorem is that the
row transfer matrices have expressions resembling vertex
operators. In particular they are $U_{\sqrt{v}}(\widehat{\mathfrak{sl}}_n)$-module
homomorphisms. This partially addresses the lack of an
interpretation of the vertical edges as
$U_{\sqrt{v}}(\widehat{\mathfrak{sl}}_n)$-modules
noted in Remark~\ref{missingmodule}.

\medbreak
\textbf{Acknowledgements}: This work was supported by NSF grants DMS-1406238
(Brubaker) and DMS-1001079, DMS-1601026 (Bump and Buciumas). We thank
Gautam Chinta, Solomon Friedberg
and Paul Gunnells for their support and encouragement, and David Kazhdan,
Daniel Orr and the referee for helpful comments.

\section{The partition function\label{partitionfunction}}

In statistical mechanics, the partition function of a model is
a \textit{generating function}. This means that through its
dependence on global parameters of the system (such as temperature)
it carries information about properties of the system such as entropy
and free energy. Here we are concerned with two-dimensional lattice
models that represent metaplectic Whittaker functions, and the global
parameters on which it depends are the Langlands parameters.


Consider a finite two-dimensional rectangular grid of fixed size, composed of interior edges
connecting to vertices of the grid and boundary edges adjacent to a single vertex
in the grid. Every edge will be assigned a
\textit{spin}, which has value $+$ or $-$. The spins along the boundary
edges will be fixed as part of the data specifying the system; the spins on
the interior edges will be allowed to vary. Thus with the spins on the
boundary fixed, a \textit{state} of the system will be an assignment of
spins to the interior edges.

We associate a system to any integer partition $\lambda = (\lambda_1, \ldots, \lambda_r)$ as follows.
The size of the grid will have $r$ rows and $N$ columns, where $N$ may be any
integer greater than or equal to $\lambda_1+r$. The boundary edge spins are
set to be $+$ at all left and bottom boundary edges, and $-$ on all right
edges. The boundary edges along the top of the grid depend on the strict
partition $\lambda + \rho$ with $\rho = (r - 1,\ldots, 3, 2, 1, 0)$.  The
spins along the top edge will be $-$ in the columns numbered
$(\lambda+\rho)_i$ for all $1 \leqslant i \leqslant r$ and $+$ on all
remaining columns.  See Figure~\ref{icestate} for an example of a state in the
system for $\lambda = (3,2,0)$ and $N = 5$.\footnote{Strictly speaking, our
systems correspond to an integer partition $\lambda$ \textit{and} the choice
of sufficiently large integer $N$ specifying the number of columns.
However, the partition function $Z(\mathfrak{S}_\lambda)$
is unchanged if we increase $N$, and we suppress~$N$ from the notation.}

Define the \textit{charge} at each horizontal edge in the configuration to
be the number of $+$ spins at or to the right of the edge, along the same row.
(This notion was introduced in~\cite{mice}.) We also will speak of the charge
at a vertex, defined to be the charge on the edge to the right of the
vertex. The charges are labeled in Figure~\ref{icestate} as decorations above
each vertex.

\begin{definition}
  The state will be called \textit{admissible} if the four spins on adjacent
  edges of any vertex are in one of the six configurations in
  Figure~\ref{gammaice} (top). It will be called
  {\textit{$n$-admissible}} if it is admissible and if furthermore every
  horizontal edge with a $-$ spin has charge $\equiv 0$ modulo~$n$.
\end{definition}

An example of an admissible state is shown in Figure~\ref{icestate}. (The
appearance of labels $z_i$ on the vertices in the figure will be explained
momentarily.) The illustrated state is $n$-admissible only if $n = 1$ or $2$,
since it has a horizontal $-$ edge with charge~$2$.

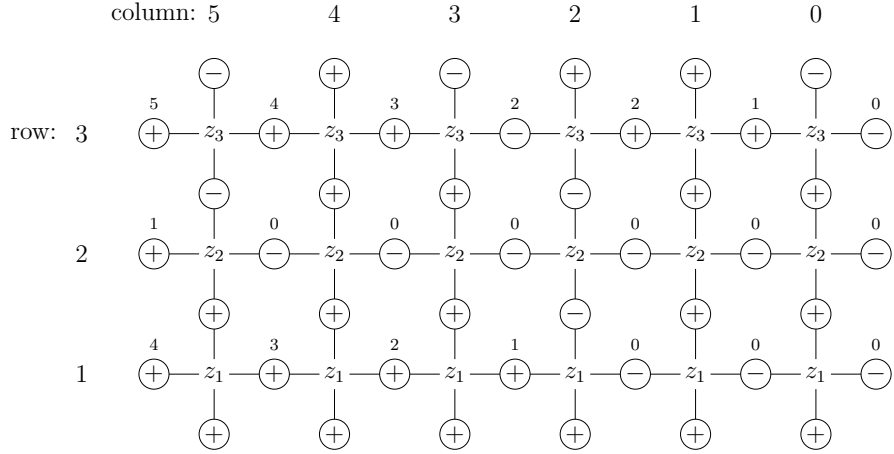
\begin{figure}
\scalebox{0.8}{
\begin{tikzpicture}
  \coordinate (ab) at (1,0);
  \coordinate (ad) at (3,0);
  \coordinate (af) at (5,0);
  \coordinate (ah) at (7,0);
  \coordinate (aj) at (9,0);
  \coordinate (al) at (11,0);
  \coordinate (ba) at (0,1);
  \coordinate (bc) at (2,1);
  \coordinate (be) at (4,1);
  \coordinate (bg) at (6,1);
  \coordinate (bi) at (8,1);
  \coordinate (bk) at (10,1);
  \coordinate (bm) at (12,1);
  \coordinate (cb) at (1,2);
  \coordinate (cd) at (3,2);
  \coordinate (cf) at (5,2);
  \coordinate (ch) at (7,2);
  \coordinate (cj) at (9,2);
  \coordinate (cl) at (11,2);
  \coordinate (da) at (0,3);
  \coordinate (dc) at (2,3);
  \coordinate (de) at (4,3);
  \coordinate (dg) at (6,3);
  \coordinate (di) at (8,3);
  \coordinate (dk) at (10,3);
  \coordinate (dm) at (12,3);
  \coordinate (eb) at (1,4);
  \coordinate (ed) at (3,4);
  \coordinate (ef) at (5,4);
  \coordinate (eh) at (7,4);
  \coordinate (ej) at (9,4);
  \coordinate (el) at (11,4);
  \coordinate (fa) at (0,5);
  \coordinate (fc) at (2,5);
  \coordinate (fe) at (4,5);
  \coordinate (fg) at (6,5);
  \coordinate (fi) at (8,5);
  \coordinate (fk) at (10,5);
  \coordinate (fm) at (12,5);
  \coordinate (gb) at (1,6);
  \coordinate (gd) at (3,6);
  \coordinate (gf) at (5,6);
  \coordinate (gh) at (7,6);
  \coordinate (gj) at (9,6);
  \coordinate (gl) at (11,6);
  \coordinate (bb) at (1,1);
  \coordinate (bd) at (3,1);
  \coordinate (bf) at (5,1);
  \coordinate (bh) at (7,1);
  \coordinate (bj) at (9,1);
  \coordinate (bl) at (11,1);
  \coordinate (db) at (1,3);
  \coordinate (dd) at (3,3);
  \coordinate (df) at (5,3);
  \coordinate (dh) at (7,3);
  \coordinate (dj) at (9,3);
  \coordinate (dl) at (11,3);
  \coordinate (fb) at (1,5);
  \coordinate (fd) at (3,5);
  \coordinate (ff) at (5,5);
  \coordinate (fh) at (7,5);
  \coordinate (fj) at (9,5);
  \coordinate (fl) at (11,5);
  \coordinate (bax) at (0,1.5);
  \coordinate (bcx) at (2,1.5);
  \coordinate (bex) at (4,1.5);
  \coordinate (bgx) at (6,1.5);
  \coordinate (bix) at (8,1.5);
  \coordinate (bkx) at (10,1.5);
  \coordinate (bmx) at (12,1.5);
  \coordinate (dax) at (0,3.5);
  \coordinate (dcx) at (2,3.5);
  \coordinate (dex) at (4,3.5);
  \coordinate (dgx) at (6,3.5);
  \coordinate (dix) at (8,3.5);
  \coordinate (dkx) at (10,3.5);
  \coordinate (dmx) at (12,3.5);
  \coordinate (fax) at (0,5.5);
  \coordinate (fcx) at (2,5.5);
  \coordinate (fex) at (4,5.5);
  \coordinate (fgx) at (6,5.5);
  \coordinate (fix) at (8,5.5);
  \coordinate (fkx) at (10,5.5);
  \coordinate (fmx) at (12,5.5);
  \draw (ab)--(gb);
  \draw (ad)--(gd);
  \draw (af)--(gf);
  \draw (ah)--(gh);
  \draw (aj)--(gj);
  \draw (al)--(gl);
  \draw (ba)--(bm);
  \draw (da)--(dm);
  \draw (fa)--(fm);
  \draw[fill=white] (ab) circle (.25);
  \draw[fill=white] (ad) circle (.25);
  \draw[fill=white] (af) circle (.25);
  \draw[fill=white] (ah) circle (.25);
  \draw[fill=white] (aj) circle (.25);
  \draw[fill=white] (al) circle (.25);
  \draw[fill=white] (ba) circle (.25);
  \draw[fill=white] (bc) circle (.25);
  \draw[fill=white] (be) circle (.25);
  \draw[fill=white] (bg) circle (.25);
  \draw[fill=white] (bi) circle (.25);
  \draw[fill=white] (bk) circle (.25);
  \draw[fill=white] (bm) circle (.25);
  \draw[fill=white] (cb) circle (.25);
  \draw[fill=white] (cd) circle (.25);
  \draw[fill=white] (cf) circle (.25);
  \draw[fill=white] (ch) circle (.25);
  \draw[fill=white] (cj) circle (.25);
  \draw[fill=white] (cl) circle (.25);
  \draw[fill=white] (da) circle (.25);
  \draw[fill=white] (dc) circle (.25);
  \draw[fill=white] (de) circle (.25);
  \draw[fill=white] (dg) circle (.25);
  \draw[fill=white] (di) circle (.25);
  \draw[fill=white] (dk) circle (.25);
  \draw[fill=white] (dm) circle (.25);
  \draw[fill=white] (eb) circle (.25);
  \draw[fill=white] (ed) circle (.25);
  \draw[fill=white] (ef) circle (.25);
  \draw[fill=white] (eh) circle (.25);
  \draw[fill=white] (ej) circle (.25);
  \draw[fill=white] (el) circle (.25);
  \draw[fill=white] (fa) circle (.25);
  \draw[fill=white] (fc) circle (.25);
  \draw[fill=white] (fe) circle (.25);
  \draw[fill=white] (fg) circle (.25);
  \draw[fill=white] (fi) circle (.25);
  \draw[fill=white] (fk) circle (.25);
  \draw[fill=white] (fm) circle (.25);
  \draw[fill=white] (gb) circle (.25);
  \draw[fill=white] (gd) circle (.25);
  \draw[fill=white] (gf) circle (.25);
  \draw[fill=white] (gh) circle (.25);
  \draw[fill=white] (gj) circle (.25);
  \draw[fill=white] (gl) circle (.25);
  \path[fill=white] (bb) circle (.25);
  \path[fill=white] (bd) circle (.25);
  \path[fill=white] (bf) circle (.25);
  \path[fill=white] (bh) circle (.25);
  \path[fill=white] (bj) circle (.25);
  \path[fill=white] (bl) circle (.25);
  \path[fill=white] (db) circle (.25);
  \path[fill=white] (dd) circle (.25);
  \path[fill=white] (df) circle (.25);
  \path[fill=white] (dh) circle (.25);
  \path[fill=white] (dj) circle (.25);
  \path[fill=white] (dl) circle (.25);
  \path[fill=white] (fb) circle (.25);
  \path[fill=white] (fd) circle (.25);
  \path[fill=white] (ff) circle (.25);
  \path[fill=white] (fh) circle (.25);
  \path[fill=white] (fj) circle (.25);
  \path[fill=white] (fl) circle (.25);
  \node at (bb) {$z_1$};
  \node at (bd) {$z_1$};
  \node at (bf) {$z_1$};
  \node at (bh) {$z_1$};
  \node at (bj) {$z_1$};
  \node at (bl) {$z_1$};
  \node at (db) {$z_2$};
  \node at (dd) {$z_2$};
  \node at (df) {$z_2$};
  \node at (dh) {$z_2$};
  \node at (dj) {$z_2$};
  \node at (dl) {$z_2$};
  \node at (fb) {$z_3$};
  \node at (fd) {$z_3$};
  \node at (ff) {$z_3$};
  \node at (fh) {$z_3$};
  \node at (fj) {$z_3$};
  \node at (fl) {$z_3$};
  \node at (-2,5) {row:};
  \node at (-1.2,5) {3};
  \node at (-1.2,3) {2};
  \node at (-1.2,1) {1};
  \node at (gb) {$-$};
  \node at (gd) {$+$};
  \node at (gf) {$-$};
  \node at (gh) {$+$};
  \node at (gj) {$+$};
  \node at (gl) {$-$};
  \node at (fa) {$+$};
  \node at (fc) {$+$};
  \node at (fe) {$+$};
  \node at (fg) {$-$};
  \node at (fi) {$+$};
  \node at (fk) {$+$};
  \node at (fm) {$-$};
  \node at (eb) {$-$};
  \node at (ed) {$+$};
  \node at (ef) {$+$};
  \node at (eh) {$-$};
  \node at (ej) {$+$};
  \node at (el) {$+$};
  \node at (da) {$+$};
  \node at (dc) {$-$};
  \node at (de) {$-$};
  \node at (dg) {$-$};
  \node at (di) {$-$};
  \node at (dk) {$-$};
  \node at (dm) {$-$};
  \node at (cb) {$+$};
  \node at (cd) {$+$};
  \node at (cf) {$+$};
  \node at (ch) {$-$};
  \node at (cj) {$+$};
  \node at (cl) {$+$};
  \node at (ba) {$+$};
  \node at (bc) {$+$};
  \node at (be) {$+$};
  \node at (bg) {$+$};
  \node at (bi) {$-$};
  \node at (bk) {$-$};
  \node at (bm) {$-$};
  \node at (ab) {$+$};
  \node at (ad) {$+$};
  \node at (af) {$+$};
  \node at (ah) {$+$};
  \node at (aj) {$+$};
  \node at (al) {$+$};
  \node at (11,7) {$0$};
  \node at (9,7) {$1$};
  \node at (7,7) {$2$};
  \node at (5,7) {$3$};
  \node at (3,7) {$4$};
  \node at (1,7) {$5$};
  \node at (0,7.04) {column:};
  \node at (bax) {$\scriptstyle 4$};
  \node at (bcx) {$\scriptstyle 3$};
  \node at (bex) {$\scriptstyle 2$};
  \node at (bgx) {$\scriptstyle 1$};
  \node at (bix) {$\scriptstyle 0$};
  \node at (bkx) {$\scriptstyle 0$};
  \node at (bmx) {$\scriptstyle 0$};
  \node at (dax) {$\scriptstyle 1$};
  \node at (dcx) {$\scriptstyle 0$};
  \node at (dex) {$\scriptstyle 0$};
  \node at (dgx) {$\scriptstyle 0$};
  \node at (dix) {$\scriptstyle 0$};
  \node at (dkx) {$\scriptstyle 0$};
  \node at (dmx) {$\scriptstyle 0$};
  \node at (fax) {$\scriptstyle 5$};
  \node at (fcx) {$\scriptstyle 4$};
  \node at (fex) {$\scriptstyle 3$};
  \node at (fgx) {$\scriptstyle 2$};
  \node at (fix) {$\scriptstyle 2$};
  \node at (fkx) {$\scriptstyle 1$};
  \node at (fmx) {$\scriptstyle 0$};
\end{tikzpicture}}
\caption{A state of a six-vertex model system. The
  columns are labeled in descending order from $N-1$
  down to $0$. The rows are labeled in descending
  order from $r$ down to $1$. In this case $N=6$, the
  partition $\lambda$ is $(3,2,0)$,
  so $\lambda+\rho=(5,3,0)$; therefore the $-$ in
  the top row are placed in columns $5$, $3$, $0$.
  The charges are shown for each horizontal edge.
  If $n=2$ this state is $n$-admissible since the
  charges of the $-$ edges are multiples of~$2$.}
\label{icestate}
\end{figure}

The
\textit{Boltzmann weight} of a state is obtained as a product of weights
attached to each vertex in the model. The weight attached to any vertex makes
use of a pair of functions $h$ and $g$ defined on the integers satisfying certain properties which we will now explain.

Let $n$ be a fixed positive integer and $v$ a fixed parameter. Let $g(a)$ be a function 
of the integer $a$ which is periodic modulo $n$, and such
that $g(0)=-v$, while $g(a)\,g(n-a)=v$ if $n$ does not divide $a$. Let 
\begin{equation}
  \label{hsumdef}
  h(a)=\left\{\begin{array}{ll}1-v&\text{if $n|a$,}\\0&\text{otherwise.}
  \end{array}\right.
\end{equation}
Choose $r$ nonzero complex numbers $z_1,\ldots,z_r$ and associate one to each row,
as indicated in Figure~\ref{icestate}. The rows are labeled $r$ down to $1$
in descending order and $z_i$ is associated with the $i$-th row as in
Figure~\ref{icestate}. Given a vertex in the $i$-th row, 
its Boltzmann weight is given in Figure~\ref{gammaice} (top). Note that this weight
depends on the spins and the charges on adjacent
edges, and the row $i$ in which it appears. Then the Boltzmann
weight $B^{(n)}_{z_1, \ldots, z_r}(\mathfrak{s})$ of the state $\mathfrak{s}$ is the product of the Boltzmann weights
over all vertices in the grid. We often omit the $n$ or the $z_1, \ldots, z_r$ in the notation for $B$, as
the weights may be stated uniformly for all such choices.

\begin{remark} \label{gausssumremark}
In~\cite{mice} and~\cite{wmd5book}, the functions $g$ and $h$ are defined
using $n$-th order Gauss sums, with $v = q^{-1}$, and shown to satisfy the above
properties. We will use this specific choice later in Theorem~\ref{matchwhit} and
in Section~\ref{connex} to connect the partition function to metaplectic Whittaker
functions. However, only the above properties are required for their study using
the Yang-Baxter equation. The function
$g(a)$ is defined in (\ref{gausssumdef}) below, and $h(a)$ we already defined
by (\ref{hsumdef}).

\end{remark}

\begin{figure}
\[
\begin{array}{|c|c|c|c|c|c|}
\hline
\tt{a}_1&\tt{a}_2&\tt{b}_1&\tt{b}_2&\tt{c}_1&\tt{c}_2\\
\hline
\begin{array}{c}\gammaice{+}{+}{+}{+}{a+1}{a}\\1\end{array} &
\begin{array}{c}\gammaice{-}{-}{-}{-}{a}{a}\\z_i\end{array} &
\begin{array}{c}\gammaice{+}{-}{+}{-}{a+1}{a}\\g(a)\end{array} &
\begin{array}{c}\gammaice{-}{+}{-}{+}{a}{a}\\z_i\end{array} &
\begin{array}{c}\gammaice{-}{+}{+}{-}{a}{a}\\h(a)z_i\end{array} &
\begin{array}{c}\gammaice{+}{-}{-}{+}{a+1}{a}\\1\end{array} \\
\hline
\hline
\begin{array}{c}\gammaice{+}{+}{+}{+}{a+1}{a}\\1\end{array} &
\begin{array}{c}\gammaice{-}{-}{-}{-}{0}{0}\\z_i\end{array} &
\begin{array}{c}\gammaice{+}{-}{+}{-}{a+1}{a}\\g(a)\end{array} &
\begin{array}{c}\gammaice{-}{+}{-}{+}{a}{a}\\z_i\end{array} &
\begin{array}{c}\gammaice{-}{+}{+}{-}{0}{0}\\(1-v)z_i\end{array} &
\begin{array}{c}\gammaice{+}{-}{-}{+}{1}{0}\\1\end{array} \\
\hline
\end{array}
\]
\caption{The Boltzmann weights at a vertex. Top: the weights in
  \textit{Metaplectic Ice}~\cite{mice}. Bottom: the weights in
  this paper. These produce the same partition function but are
  subtly different in that the new weights satisfy a Yang-Baxter
  equation as in Theorem~\ref{gammaybe}. The illustrated
  vertices are in the $i$-th row and have charge $a$. (The charge is the
  number of $+$ signs in the row to the right of the vertex.) The
  Boltzmann weight of any configuration not appearing in this
  table is zero. In an $n$-admissible state, any horizontal edge with a $-$ spin 
  will have its charge divisible by $n$. Each admissible
  configuration is assigned a \textit{type}
  $\tt{a}_1$, $\tt{a}_2$,
  $\tt{b}_1$, $\tt{b}_2$,
  $\tt{c}_1$ or $\tt{c}_2$.}
\label{gammaice}
\end{figure}

\begin{example}
In the state in Figure~\ref{icestate} we look at the top row. Using the classification
of admissible configurations in Figure~\ref{gammaice} (top), the vertex in column $5$ is of
type $\tt{b}_1$ with charge $4$, so its Boltzmann weight is
$g(4)$. There are two vertices of type $\tt{c}_2$
in columns $3$ and $0$, a $\tt{c}_1$
vertex in column $2$, and $\tt{a}_1$ vertices in columns $4$ and $1$. 
The $\tt{c}_1$ vertex and one of the $\tt{c}_2$ vertices have charge
$2$, so the state is $n$-admissible only if $n=1$ or $2$.  Assuming
this, the $g(4)$ from the $\tt{b}_1$ vertex evaluates to $-v$ and the
$\tt{c}_1$ vertex evaluates to $(1-v)z_1$, while the remaining vertices in
the row have weight $1$. Thus the total contribution of this row to the weight of the state is
$(-v)(1-v)z_1$.  The second row has a vertex $\tt{c}_2$ (with charge 0)
and the remaining vertices are all of type $\tt{b}_2$ or $\tt{a}_2$, so this row
contributes $z_2^5$. The last row has a $\tt{c}_2$ vertex (with charge
$0$) and two $\tt{b}_2$ vertices. The remaining three vertices in the row
are of type $\tt{a}_1$ with Boltzmann weight $1$. The Boltzmann weight of
this state is $(-v)(1-v)z_1z_2^5z_3^2$ if $n=1$ or $2$, and $0$
otherwise.
\end{example}

\begin{proposition}
\label{c2lemma}
Suppose that $\mathfrak{s}$ is an admissible state such that
the Boltzmann weight $B^{(n)}(\mathfrak{s})\neq0$. Then the state is
$n$-admissible.
\end{proposition}

\begin{proof}
We must show that, under this assumption, the charge on every edge with spin $-$
is a multiple of~$n$. Suppose not and consider the right-most vertex in any row where this
condition fails; that is, the charge $a$ of the edge to the right is not a multiple of $n$.
We claim that the edge to the right of $v$ has spin $+$.
We know that $v$ is not the rightmost vertex in its row
since its charge is nonzero. So if the edge to the right
of $v$ has spin $-$, then the vertex to the right of $v$
has the same charge as $v$, contradicting our assumption
that $v$ is the rightmost counterexample in its row.

Since the edge to the left of $v$ is $-$ and the
edge to the right is $+$, consulting Figure~\ref{gammaice} (top)
we see that the only admissible configuration of spins at the vertex
$v$ is of type $\tt{c}_1$, so the Boltzmann weight
at $v$ is $h(a)\,z_i=0$ because $n\nmid a$. This
contradicts our assumption that $B^{(n)}(\mathfrak{s})\neq 0$.
\end{proof}

We may now explain the distinction between the system
in~\cite{mice} and the one used throughout this paper. Let $\mathfrak{S}'_\lambda$ denote the
set of admissible states, and let $\mathfrak{S}_\lambda$ denote the
smaller set of $n$-admissible states. In~\cite{mice}, the
partition function $Z(\mathfrak{S}'_\lambda)$ is defined
to be the sum of $B^{(n)}(\mathfrak{s})$ where $\mathfrak{s}$ runs over
$\mathfrak{S}'_\lambda$. In this paper, we consider
partition function $Z(\mathfrak{S}_\lambda)$, the sum over
$n$-admissible states.

\begin{theorem} \label{matchwhit}
  Let $\lambda$ be a partition with $r$ parts and let $n$ be a fixed positive integer. Then
  the partition function $Z(\mathfrak{S}_\lambda)$ is 
  (up to normalization) a value of a $p$-adic spherical Whittaker function on the
  metaplectic $n$-fold cover of $\GL(r,F)$, where $F$ is a nonarchimedean
  local field with residue field of cardinality $q \equiv 1 \; (2n)$.
\end{theorem}

\begin{proof}
  By Proposition~\ref{c2lemma},
  $Z(\mathfrak{S}'_\lambda)=Z(\mathfrak{S}_\lambda)$. Combining this with
  Theorem~4 of~\cite{mice}, the statement follows.
\end{proof}

In this result, $g(a)$ and $h(a)$ appearing in the Boltzmann weights for $Z$
are $n$-th order Gauss sums as explained in Remark~\ref{gausssumremark} and
$v= q^{-1}$ where $q$ is the cardinality of the residue field of $F$.

\begin{remark} Theorem~4 of \cite{mice} depends on the crystal description
  of the Type~A Whittaker functions proved in \cite{McNamaraDuke}. (See also
  \cite{eisenxtal}.)  However \cite{McNamaraDuke} treats simple groups such as
  $\SL(r)$. This apparent gap can be remedied by noting that the metaplectic
  Casselman-Shalika formula proved for reductive groups in \cite{McNamaraCS}
  is equivalent to the needed crystal description by results of
  Puskas~\cite{Puskas}.
\end{remark}

\begin{remark} The normalization that is unspecified in this statement will
  be made precise in (\ref{thmoneprec}). Moreover,
  Theorem~\ref{whitpartcharge} is a generalization of this
  result that expresses a basis of all $n^r$ spherical Whittaker functions
  as partition functions of Gamma ice, and also elucidates the relationship
  with Theorem~\ref{commdiagwithintertwiner}.
\end{remark}

To summarize, we may strictly limit the admissible states so
that $-a$ only occurs with the charge $a\equiv 0$ modulo~$n$.
Thus we may use the Boltzmann weights of Figure~\ref{gammaice} (bottom). The
restriction to $n$-admissible states does not change the partition function
but has the benefit of making the model solvable in the sense of Baxter. This
means that it is amenable to study by the Yang-Baxter equation.

\section{The Yang-Baxter equation\label{ybesection}}
The partition functions described in Section~\ref{partitionfunction}
differ from those of the classical six-vertex model in a crucial way: the
Boltzmann weights depend on a global statistic, the charge. If we wish
to use statistical mechanical techniques like the Yang-Baxter equation, we
need the weight at any vertex to be \textit{local}, that is, depending only on
nearest-neighbor interactions. We achieve this by a slight change in point of view, 
introducing \textit{decorated spins} for the horizontal edges. Given a fixed positive integer $n$,
a \textit{decorated spin} is an ordered
pair $(\sigma,a)$ where the \textit{spin} $\sigma$ is $+$ or $-$ and the \textit{decoration} $a$ is
an integer mod $n$. Moreover if $\sigma=-$, we will only consider $a \equiv 0$ mod $n$. 
In figures we will sometimes draw the spin $\sigma$ in a circle and
write the decoration $a$ next to it. In text we will denote $(\sigma,a)$ as
$\sigma a$. The key point is that the decoration is now viewed as part of the data attached to
a horizontal edge. Now there are $n+1$ possible decorated spins for horizontal edges, 
rather than just the spins $+$ and $-$; we have left the six-vertex model.

Not all choices of decorated spins on horizontal edges will have nonzero Boltzmann weight. Each decoration $a$ on
a horizontal edge to the left of a vertex must be compatible with its spin $\sigma$ and the decoration $b$ on the edge 
to the right. If $\sigma = +$, then $a \equiv b+1$ (mod $n$) and 
if $\sigma = -$, then $a \equiv b$. If we set the initial decorations of the right-hand boundary edges 
(which all have spin $-$) to be 0, then this rule clearly recovers the charge (mod $n$) of the previous section. 
Thus the Boltzmann weights in Figure~\ref{gammaice} may be interpreted as purely
local; in the figure, we have indicated the decoration by writing
it over the spin. We are justified in requiring the
decoration at a $-$ edge to be $0$ modulo~$n$ (without affecting the resulting partition functions) by
Proposition~\ref{c2lemma}.


Now that the weights at any vertex may be viewed as local, we are ready to present
our solution to the Yang-Baxter equation. Three sets of vertices with different Boltzmann 
weights will appear in the Yang-Baxter equation (Theorem~\ref{gammaybe}) below. In figures, these will be 
labeled $z_i$, $z_j$ and $R_{z_i,z_j}$. Here $z_i$ and $z_j$ are nonzero complex 
numbers used in the Boltzmann weight of the associated vertex. At the vertices with labels 
$z_i$ and $z_j$ we will use the Boltzmann weights already described in Figure~\ref{gammaice}.
In Figure~\ref{gamgamice} we describe the Boltzmann weights at the vertices
labeled $R_{z_i,z_j}$.  The Boltzmann weights depend only on the residue
classes modulo $n$ of the integers $a,b,c,\ldots$ that appear in these
formulas, but in some cases depend on a particular choice of representatives
for residue classes. These choices are indicated in the description below the figure.

\begin{figure}[h]
\[\begin{array}{|c|c|c|c|c|c|}\hline
\begin{array}{c}\gamgam{+}{+}{+}{+}{a}{a}{a}{a}\\z_i^n-vz_j^n\end{array} & 
\begin{array}{c}\gamgam{+}{+}{+}{+}{b}{a}{b}{a}\\g(a-b)(z_j^n-z_i^n)\end{array} & 
\begin{array}{c}\gamgam{+}{+}{+}{+}{b}{a}{a}{b}\\(1-v)z_j^cz_i^{n-c}\;\text{(*)}\end{array} & 
\begin{array}{c}\gamgam{-}{-}{-}{-}{0}{0}{0}{0}\\z_j^n-vz_i^n\end{array} \\ 
\hline
\begin{array}{c}\gamgam{+}{-}{+}{-}{a}{0}{a}{0}\\v(z_j^n - z_{i}^{n})\end{array} & 
\begin{array}{c}\gamgam{-}{+}{-}{+}{0}{a}{0}{a}\\z_j^n - z_{i}^{n}\end{array} & 
\begin{array}{c}\gamgam{-}{+}{+}{-}{0}{a}{a}{0}\\(1-v)z_j^az_i^{n-a}\;\text{(**)}\end{array} & 
\begin{array}{c}\gamgam{+}{-}{-}{+}{a}{0}{0}{a}\\(1-v)z_j^{n-a}z_i^a\;\text{(**)}\end{array} \\ 
\hline\end{array}
\]
\caption{\parindent=0pt\parskip=0pt
Boltzmann weights for the $R$-vertex $R_{z_i,z_j}$. It is assumed that $b$ is
not equal to $a$.
(*) Here $c\equiv a-b$ mod $n$ with $0\leqslant c<n$.
(**) Here we choose the representative of $a$ modulo $n$ with $1\leqslant
  a\leqslant n$, so if $a\equiv0$ mod $n$, $n-a$ means $0$, not~$n$.}
\label{gamgamice}
\end{figure}

\begin{theorem}
\label{gammaybe}
The partition functions of the following two systems are
equal. That is, if we fix the charges $\sigma$, $\tau$, $\beta$, $\rho$, $\alpha$ and $\theta$
and the decorations $a,b,c,d$, and sum over all possible values of the inner
edge (decorated) spins, we obtain the same result in both cases.
 \vglue -20pt
\begin{equation}\label{mybe}
\begin{array}{c} \lhs{\botcharge{\sigma}{a}}{\topcharge{\tau}{b}}{\beta}{\topcharge{\theta}{c}}{\botcharge{\rho}{d}}{\alpha}{\topcharge{\nu}{e}}{\botcharge{\mu}{f}}{\gamma} \end{array}
\qquad \qquad
\begin{array}{c} \rhs{\botcharge{\sigma}{a}}{\topcharge{\tau}{b}}{\beta}{\topcharge{\theta}{c}}{\botcharge{\rho}{d}}{\alpha}{\botcharge{\phi}{g}}{\topcharge{\psi}{h}}{\delta} \end{array}
\end{equation}
 \vglue 15pt
\end{theorem}

\begin{proof}
In every admissible configuration
there are an even number of $+$ spins on the six boundary
edges. Therefore there are 32 possible boundary spin choices,
and we must consider each of these cases separately. Moreover, each case
breaks into subcases depending on the decorations at horizontal boundary edges
with spins $\sigma,\tau,\theta$ and $\rho$. To give the reader a feeling for the
possibilities, we will do one case in detail.
The remaining cases may be found in~\cite{ThisPaper}.

We will consider Case~10, using the enumeration of cases in~\cite{ThisPaper}, whose assignment
of boundary spins is 
$(\sigma,\tau,\beta,\theta,\rho,\alpha)=(+,+,-,+,-,+)$.

%
\medbreak\noindent
\textbf{Case 10a}: With $k\neq 0$, suppose that the (decorated)
spins on the six boundary edges are as follows:
\[ \begin{array}{|c|c|c|c|c|c|}
     \hline
     \begin{array}{c}a\\ \sigma\end{array} & 
     \begin{array}{c}b\\ \tau\end{array} & 
     \begin{array}{c}\\ \beta\end{array} & 
     \begin{array}{c}c\\ \theta\end{array} & 
     \begin{array}{c}d\\ \rho \end{array} & 
     \begin{array}{c}\\ \alpha\end{array}\\ 
     \hline
     \begin{array}{c} \scriptstyle k + 1\\ + \end{array} & 
     \begin{array}{c} 1\\ + \end{array} & 
     \begin{array}{c} \\ - \end{array} & 
     \begin{array}{c} k\\ + \end{array} & 
     \begin{array}{c} 0\\ - \end{array} & 
     \begin{array}{c} \\  + \end{array}\\
     \hline
   \end{array} \]
On each side there is one $n$-admissible state:
\[\begin{array}{cc}\text{left hand side}&\text{right hand side}\\
     \begin{array}{|c|c|c|c|}
     \hline
     \begin{array}{c} e\\\nu \end{array} & \begin{array}{c} f\\ \mu 
     \end{array} & \begin{array}{c}\\ \gamma \end{array} & 
     \text{weight}\\\hline
     \begin{array}{c} k + 1\\ + \end{array} & 
     \begin{array}{c} 1\\ + \end{array} & 
     \begin{array}{c} \\ - \end{array} & 
     (z_{j}^{n}-z_i^n)\,g(k)\,g(-k)\\
     \hline
   \end{array}
 &
     \begin{array}{|c|c|c|c|}
     \hline
     \begin{array}{c}g\\\phi\end{array} & 
     \begin{array}{c} h\\ \psi \end{array} & 
     \begin{array}{c} \\ \delta\end{array} & \text{weight}\\
     \hline
     \begin{array}{c} k\\+ \end{array} & 
     \begin{array}{c} 0\\ -\end{array} & 
     \begin{array}{c} \\ +\end{array} &
     (z_{j}^{n}-z_i^n)\,v\\
     \hline
   \end{array}
\end{array}\;.\]
Thus the configurations are as follows:
 \vglue -15pt
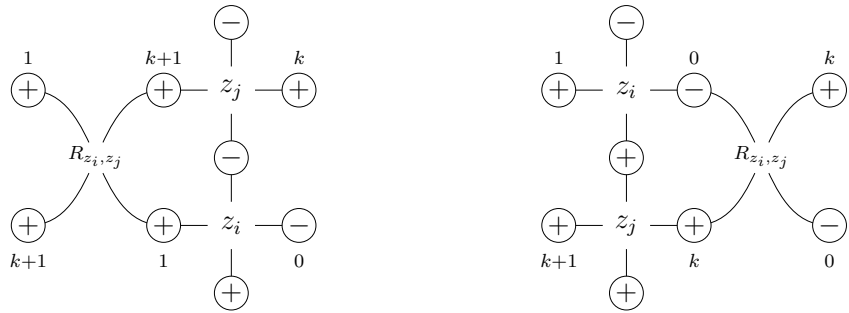
\[
\lhs{\botcharge{+}{k + 1}}{\topcharge{+}{1}}{-}{\topcharge{+}{k}}{\botcharge{-}{0}}{+}{\topcharge{+}{k + 1}}{\botcharge{+}{1}}{-}
\qquad\qquad
\rhs{\botcharge{+}{k + 1}}{\topcharge{+}{1}}{-}{\topcharge{+}{k}}{\botcharge{-}{0}}{+}{\botcharge{+}{k}}{\topcharge{-}{0}}{+}\]
\vglue 15pt
Since $g(k)\,g(-k)=v$, (\ref{mybe}) is satisfied in this case.

\medbreak\noindent
\textbf{Case 10b}: With $k\neq 0$,
\[ \begin{array}{|c|c|c|c|c|c|}
     \hline
     \begin{array}{c}a\\ \sigma\end{array} & 
     \begin{array}{c}b\\ \tau\end{array} & 
     \begin{array}{c}\\ \beta\end{array} & 
     \begin{array}{c}c\\ \theta\end{array} & 
     \begin{array}{c}d\\ \rho \end{array} & 
     \begin{array}{c}\\ \alpha\end{array}\\ 
     \hline
     \begin{array}{c} 1\\ + \end{array} & 
     \begin{array}{c} \scriptstyle k+1\\ + \end{array} & 
     \begin{array}{c} \\ - \end{array} & 
     \begin{array}{c} k\\ + \end{array} & 
     \begin{array}{c} 0\\ - \end{array} & 
     \begin{array}{c} \\  + \end{array}\\
     \hline
   \end{array} \]

\[\begin{array}{cc}\text{left hand side}&\text{right hand side}\\
     \begin{array}{|c|c|c|c|}
     \hline
     \begin{array}{c} e\\\nu \end{array} & \begin{array}{c} f\\ \mu 
     \end{array} & \begin{array}{c}\\ \gamma \end{array} & 
     \text{weight}\\\hline
     \begin{array}{c} \scriptstyle k+1\\ + \end{array} & 
     \begin{array}{c} 1\\ + \end{array} & 
     \begin{array}{c} \\ - \end{array} & 
     (1-v)\,z_j^k\,z_i^{n-k}\,g(k)\\
     \hline
   \end{array}
 &
     \begin{array}{|c|c|c|c|}
     \hline
     \begin{array}{c}g\\\phi\end{array} & 
     \begin{array}{c} h\\ \psi \end{array} & 
     \begin{array}{c} \\ \delta\end{array} & \text{weight}\\
     \hline
     \begin{array}{c} 0\\- \end{array} & 
     \begin{array}{c} k\\ +\end{array} & 
     \begin{array}{c} \\ -\end{array} &
     (1-v)\,z_j^k\,z_i^{n-k}\,g(k)\\
     \hline
   \end{array}
\end{array}\]

\medbreak\noindent
\textbf{Case 10c}:
\[ \begin{array}{|c|c|c|c|c|c|}
     \hline
     \begin{array}{c}a\\ \sigma\end{array} & 
     \begin{array}{c}b\\ \tau\end{array} & 
     \begin{array}{c}\\ \beta\end{array} & 
     \begin{array}{c}c\\ \theta\end{array} & 
     \begin{array}{c}d\\ \rho \end{array} & 
     \begin{array}{c}\\ \alpha\end{array}\\ 
     \hline
     \begin{array}{c} 1\\ + \end{array} & 
     \begin{array}{c} 1\\ + \end{array} & 
     \begin{array}{c} \\ - \end{array} & 
     \begin{array}{c} 0\\ + \end{array} & 
     \begin{array}{c} 0\\ - \end{array} & 
     \begin{array}{c} \\  + \end{array}\\
     \hline
   \end{array} \]

\[\begin{array}{cc}\text{left hand side}&\text{right hand side}\\
     \begin{array}{|c|c|c|c|}
     \hline
     \begin{array}{c} e\\\nu \end{array} & \begin{array}{c} f\\ \mu 
     \end{array} & \begin{array}{c}\\ \gamma \end{array} & 
     \text{weight}\\\hline
     \begin{array}{c} 1\\ + \end{array} & 
     \begin{array}{c} 1\\ + \end{array} & 
     \begin{array}{c} \\ - \end{array} & 
     v(z_j^n\,v - z_i^n)\\
     \hline
   \end{array}
 &
     \begin{array}{|c|c|c|c|}
     \hline
     \begin{array}{c}g\\\phi\end{array} & 
     \begin{array}{c} h\\ \psi \end{array} & 
     \begin{array}{c} \\ \delta\end{array} & \text{weight}\\
     \hline
     \begin{array}{c} 0\\- \end{array} & 
     \begin{array}{c} 0\\ +\end{array} & 
     \begin{array}{c} \\ -\end{array} &
     (v-1)\,v\,z_j^n\\
     \hline
     \begin{array}{c} 0\\+ \end{array} & 
     \begin{array}{c} 0\\ -\end{array} & 
     \begin{array}{c} \\ +\end{array} &
     (z_j^n-z_i^n)\,v\\
     \hline
   \end{array}
\end{array}\]

This exhausts all possible choices of decorations on boundary edges, and hence completes the proof of Case 10. See the appendix in~\cite{ThisPaper}
for the other cases.
\end{proof}

There is another Yang-Baxter equation to be mentioned.

\begin{theorem}
\label{firstbraid}
Let $z_i$, $z_j$ and $z_k$ be given. Then for every
choice of decorated boundary spins $\alpha, \beta, \gamma, \delta, \epsilon,\phi$,
the partition functions of the following two systems are equal:
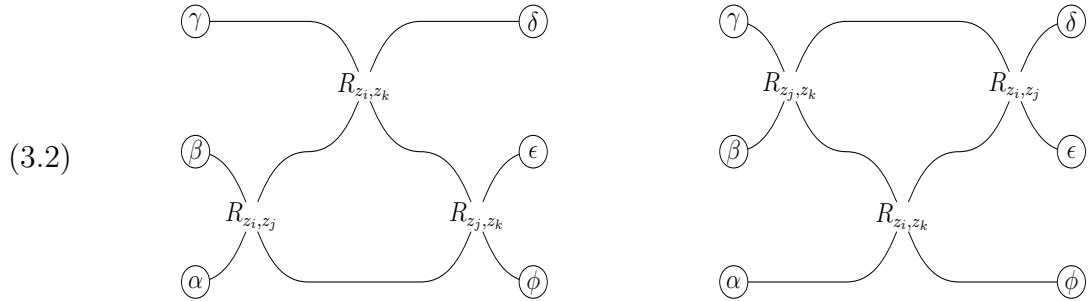
\begin{equation} \label{paramybe} \begin{array}{c} \resizebox{5cm}{4cm}{
\begin{tikzpicture}
\coordinate (a1) at (0,0);
\coordinate (b1) at (0,2);
\coordinate (c1) at (0,4);
\coordinate (a2) at (2,0);
\coordinate (b2) at (2,2);
\coordinate (c2) at (2,4);
\coordinate (a3) at (4,0);
\coordinate (b3) at (4,2);
\coordinate (c3) at (4,4);
\coordinate (a4) at (6,0);
\coordinate (b4) at (6,2);
\coordinate (c4) at (6,4);
\coordinate (r) at (3,3);
\coordinate (s) at (1,1);
\coordinate (t) at (5,1);
\draw (a1) to [out=0,in=180] (b2) to [out=0,in=180] (c3) to [out=0,in=180] (c4);
\draw (b1) to [out=0,in=180] (a2) to [out=0,in=180] (a3) to [out=0,in=180] (b4);
\draw (c1) to [out=0,in=180] (c2) to [out=0,in=180] (b3) to [out=0,in=180] (a4);
\draw[fill=white] (a1) circle (.25);
\draw[fill=white] (b1) circle (.25);
\draw[fill=white] (c1) circle (.25);
\draw[fill=white] (a4) circle (.25);
\draw[fill=white] (b4) circle (.25);
\draw[fill=white] (c4) circle (.25);
\path[fill=white] (r) circle (.25);
\path[fill=white] (s) circle (.25);
\path[fill=white] (t) circle (.25);
\node at (t) {$R_{z_j,z_k}$};
\node at (r) {$R_{z_i,z_k}$};
\node at (s) {$R_{z_i,z_j}$};
\node at (a1) {$\alpha$};
\node at (b1) {$\beta$};
\node at (c1) {$\gamma$};
\node at (c4) {$\delta$};
\node at (b4) {$\epsilon$};
\node at (a4) {$\phi$};
\end{tikzpicture}} \end{array} \; \qquad\qquad
\begin{array}{c} \resizebox{5cm}{4cm}{
\begin{tikzpicture}
\coordinate (a1) at (0,0);
\coordinate (b1) at (0,2);
\coordinate (c1) at (0,4);
\coordinate (a2) at (2,0);
\coordinate (b2) at (2,2);
\coordinate (c2) at (2,4);
\coordinate (a3) at (4,0);
\coordinate (b3) at (4,2);
\coordinate (c3) at (4,4);
\coordinate (a4) at (6,0);
\coordinate (b4) at (6,2);
\coordinate (c4) at (6,4);
\coordinate (r) at (3,1);
\coordinate (s) at (1,3);
\coordinate (t) at (5,3);
\draw (c1) to [out=0,in=180] (b2) to [out=0,in=180] (a3) to [out=0,in=180] (a4);
\draw (b1) to [out=0,in=180] (c2) to [out=0,in=180] (c3) to [out=0,in=180] (b4);
\draw (a1) to [out=0,in=180] (a2) to [out=0,in=180] (b3) to [out=0,in=180] (c4);
\draw[fill=white] (a1) circle (.25);
\draw[fill=white] (b1) circle (.25);
\draw[fill=white] (c1) circle (.25);
\draw[fill=white] (a4) circle (.25);
\draw[fill=white] (b4) circle (.25);
\draw[fill=white] (c4) circle (.25);
\path[fill=white] (r) circle (.25);
\path[fill=white] (s) circle (.25);
\path[fill=white] (t) circle (.25);
\node at (s) {$R_{z_j,z_k}$};
\node at (r) {$R_{z_i,z_k}$};
\node at (t) {$R_{z_i,z_j}$};
\node at (a1) {$\alpha$};
\node at (b1) {$\beta$};
\node at (c1) {$\gamma$};
\node at (c4) {$\delta$};
\node at (b4) {$\epsilon$};
\node at (a4) {$\phi$};
\end{tikzpicture}} \end{array}
\end{equation}
\end{theorem}

Theorem~\ref{firstbraid} presents a parametrized
Yang-Baxter equation.
We will eventually relate this to the parametrized Yang-Baxter equation
associated with the $R$-matrix of a quantum group, namely a Drinfeld
twist of $\widehat{\mathfrak{gl}}(1|n)$.

\begin{proposition}
\label{secondbraid}
Let $\alpha,\beta,\gamma,\delta$ be decorated spins. Then the partition function of
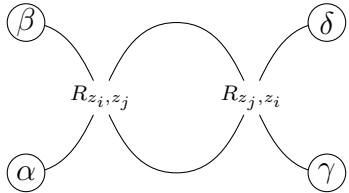
\[\begin{tikzpicture}
\coordinate (a1) at (0,0);
\coordinate (b1) at (0,2);
\coordinate (a2) at (2,0);
\coordinate (b2) at (2,2);
\coordinate (a3) at (4,0);
\coordinate (b3) at (4,2);
\coordinate (r) at (1,1);
\coordinate (s) at (3,1);
\draw (a1) to [out=0,in=180] (b2) to [out=0,in=180] (a3);
\draw (b1) to [out=0,in=180] (a2) to [out=0,in=180] (b3);
\draw[fill=white] (a1) circle (.25);
\draw[fill=white] (b1) circle (.25);
\draw[fill=white] (a3) circle (.25);
\draw[fill=white] (b3) circle (.25);
\path[fill=white] (r) circle (.25);
\path[fill=white] (s) circle (.25);
\node at (r) {$\scriptstyle R_{z_i,z_j}$};
\node at (s) {$\scriptstyle R_{z_j,z_i}$};
\node at (a1) {$\alpha$};
\node at (b1) {$\beta$};
\node at (a3) {$\gamma$};
\node at (b3) {$\delta$};
\end{tikzpicture}\]
equals
\[\left\{\begin{array}{ll}(z_j^n-vz_i^n)(z_i^n-vz_j^n)&\text{if $\alpha=\gamma$,
  $\beta=\delta$}\\ 0 &\text{otherwise.}
\end{array}\right.\]
\end{proposition}  

As we mentioned in the introduction, we will show in
Section~\ref{connex} that Theorem~\ref{firstbraid} is related to
the intertwining integrals for principal series representations
of the metaplectic group, which were calculated in Kazhdan
and Patterson~\cite{KazhdanPatterson}. From this point of
view, Proposition~\ref{secondbraid} is related to Theorem~I.2.6
of~\cite{KazhdanPatterson}.

Because of the last result, it is \textit{almost} true that if we
modified the $R$-matrix $R_{z_i,z_j}$ by dividing by $z_j^n-vz_i^n$, the
associated quantum (super) group (which will be identified in
Section~\ref{supersection}) would be \textit{triangular} in the sense of
Drinfeld~\cite{Drinfeld}.  However because this factor $z_j^n-vz_i^n$ can
be zero, this is not quite true, and the braided category of modules is also not
triangular.

\begin{proof}[Proofs of Theorem~\ref{firstbraid} and Proposition~\ref{secondbraid}.]
The earlier version~\cite{ThisPaper} contains a proof that
Theorem~\ref{firstbraid} and Proposition~\ref{secondbraid} follow
from Theorem~\ref{gammaybe}. However Theorem~\ref{firstbraid}
may be proved more straightforwardly along the lines of
Theorem~\ref{gammaybe} by consideration of the different
cases, or deduced from Kojima~\cite{KojimaChain} equation~(2.13).
We will verify at the end of the next section that Kojima's
Yang-Baxter equation implies Theorem~\ref{firstbraid}, and
Proposition~\ref{secondbraid} is straightforward.
\end{proof}

\section{Metaplectic Ice and Supersymmetry\label{supersection}}
Perk and Schultz~\cite{PerkSchultz} found new solutions
of the Yang-Baxter equation. Meanwhile graded (supersymmetric)
Yang-Baxter equations were introduced by Bazh\-anov and
Shadrikov~\cite{BazhanovShadrikov}.  It was found by
Yamane~\cite{YamaneSuper} that the Perk-Schultz equations
were related to the $R$-matrix of the quantized enveloping
algebra of the $\mathfrak{gl}(m|n)$ Lie superalgebra in
the standard representation.  The quantized enveloping
algebra of the corresponding affine Lie superalgebra was
considered by Zhang~\cite{ZhangRBSupergroup}. A convenient
reference for us is Kojima~\cite{KojimaChain}. See also~\cite{ZhangFundamental}. 

We will explain how to relate the $R$-vertex weights of the prior section,
which depend on a fixed $n$ as in Figure~\ref{gamgamice}, 
to the $\mathfrak{gl}(1|n)$ $R$-matrix. The relationship
is rather subtle, since we will have to
perform manipulations on the Perk-Schultz $R$-matrix
in order to make the comparison. These manipulations
preserve the Yang-Baxter equation as in Theorem~\ref{firstbraid}, but (among other
things) they introduce $n$-th order Gauss sums which are crucial in
the connection to representation theory of the metaplectic group \cite{KazhdanPatterson,mice,wmd5book}.

If $V$ is the $(1|n)$-dimensional defining module of quantum $\mathfrak{gl}(1|n)$, then for every $z \in \mathbb{C}^\times$ there is an evaluation module $V_z$ of $U_{q}(\widehat{\mathfrak{gl}}(1|n))$. One can associate $R$-matrices $R_{z_i,z_j} \in \text{End}(V_{z_i} \otimes V_{z_j})$ (see \cite{KojimaChain} for more details) that satisfy a graded Yang-Baxter equation in $\operatorname{End}(V_{z_i} \otimes V_{z_j} \otimes V_{z_k})$. As noted in~\cite{KojimaChain},
we may change some signs in the $R$-matrix 
to produce a solution to the \textit{ungraded} Yang-Baxter equation; this is the $R$-matrix we wish to compare with that in Theorem~\ref{firstbraid}. A basis
of the ungraded $(1|n)$-dimensional vector space $V_z$ can
be taken to be the decorated edge spins $-0$
for the even part, and $+a$ with $a$ modulo $n$
for the odd part, of a vertex with parameter $z$. 

Referring to \cite{KojimaChain} for notation, we will
take the decorated spin $-0$ to have graded degree~0,
and the spins $+a$, where $a$ is an integer modulo $n$,
to have degree~1. Thus we are concerned with $\widehat{\mathfrak{gl}}(1|n)$.
For the sake of comparing our results to Kojima's,
the parameter $q$ in this section will be Kojima's $q$,
which will equal $\sqrt{v}$; it is not the same as
$q$ (the cardinality of the $p$-adic residue field) in the other sections of this paper.

In Figure~\ref{comparekojima} we have the Boltzmann
weights from Figure~\ref{gamgamice} divided by
$z_i^n$, compared with the corresponding $R$-matrix entries
from (2.4)-(2.7) of~\cite{KojimaChain} which we have multiplied
by the constant $1-q^2z$.

We will give two ways of modifying the $R$-matrix to obtain another $R$-matrix
that is also a solution of the Yang-Baxter equation. One method only
affects the weights in cases III, VII and VIII. The other only affects
the weights in cases II, V and VI. After these changes, we will be able
to match the Kojima Boltzmann weights up to sign, with $z=z_i/z_j$
and $q^2=v$. (Then we will have to discuss the sign.)

\begin{figure}[h]
\[ \begin{array}{|l|l|l||l|l|l|}
     \hline
     & \text{This paper} & \text{Kojima}&&\text{This paper}&\text{Kojima}\\
     \hline
\begin{array}{l}\small I.\\\vcenter{\hbox
    to.7in{\gamgam{+}{+}{+}{+}{a}{a}{a}{a}}}
  \end{array}
&\small (z_i/z_j)^n - v &\small z-q^2 &
\begin{array}{l}\small V.\\\vcenter{\hbox to.7in{\gamgam{+}{-}{+}{-}{a}{0}{a}{0}}}\end{array}
&\small v (1 - (z_i/z_j)^n) &\small q(1-z)\\
     \hline
\begin{array}{l}\small II.\\\vcenter{\hbox to.7in{\gamgam{+}{+}{+}{+}{b}{a}{b}{a}}}\end{array}
& \small g (a - b) (1 - (z_i/z_j)^n) & \small q(1-z) &
\begin{array}{l}\small VI.\\ \vcenter{\hbox to.7in{\gamgam{-}{+}{-}{+}{0}{a}{0}{a}}}\end{array}
&\small 1 - (z_i/z_j)^n &\small q(1-z)\\
     \hline
\begin{array}{l}\small III.\\\vcenter{\hbox to.7in{\gamgam{+}{+}{+}{+}{b}{a}{a}{b}}}\end{array}
& 
\begin{array}{l}
  \small (1-v) (z_i/z_j)^{n-a+b} \\ \\
  \small (1-v) (z_i/z_j)^{-a+b}
     \end{array}
& 
\small \begin{array}{l}
  \small z(q^2 - 1) \\ \small \text{if} \; a > b,\\ \\
  \small (q^2 - 1) \\ \small \text{if} \; a < b
     \end{array} &
\begin{array}{l}\small VII.\\\vcenter{\hbox to.7in{\gamgam{-}{+}{+}{-}{0}{a}{a}{0}}}\end{array}
&\small (1 - v) (z_i/z_j)^{n - a} &\small z(1-q^2)\\
     \hline
\begin{array}{l}\small IV.\\\vcenter{\hbox to.7in{\gamgam{-}{-}{-}{-}{0}{0}{0}{0}}}\end{array}
&\small 1 - v (z_i/z_j)^n &\small q^2z-1 &
\begin{array}{l}\small VIII.\\\vcenter{\hbox to.7in{\gamgam{+}{-}{-}{+}{a}{0}{0}{a}}}\end{array}
&\small (1 - v) (z_i/z_j)^a &\small 1-q^2\\
     \hline
   \end{array} \]
\caption{Left: The $R$-matrix from Figure~\ref{gamgamice},
  divided by $z_j^n$. Right: the Boltzmann weights of
  Kojima's $\widehat{\mathfrak{g}\mathfrak{l}}(1|n)$
  $R$-matrix multiplied by $1-q^2z$. Just taking the first
  three cases (and discarding any case with a decorated spin $-0$)
  gives the $\widehat{\mathfrak{g}\mathfrak{l}}(n)$ $R$-matrix.
  It is assumed that $a\not\equiv b$ mod~$n$.
}
\label{comparekojima}
\end{figure}

For each nonzero complex number $z$, let $V (z)$ be an $(n + 1)$-dimensional
vector space with basis $v_{\alpha} = v_{\alpha} (z)$, where $\alpha$ runs
through $n + 1$ ``decorated spins.'' These are the ordered pairs $+a$ with
$0 \leqslant a < n$ and $-0$.

Previously we interpreted the vertex $R_{z_i, z_j}$ as a vertex in a graph
with certain Boltzmann weights attached to it. We now reinterpret it as an
endomorphism of a vector space, as usual in the application of quantum groups
to solvable lattice models. If $\alpha, \beta, \gamma, \delta$ are decorated
spins, let $R_{\alpha, \beta}^{\gamma, \delta} (z_i, z_j)$ be the Boltzmann
weight of vertex $R_{z_i, z_j}$ with the decorated spins $\alpha, \beta,
\gamma, \delta$ arranged as follows:
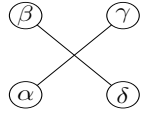
\[\gamgam{\alpha}{\beta}{\gamma}{\delta}{}{}{}{}\]
We assemble these into an endomorphism $R_{z_i, z_j}$ of $V (z_1) \otimes V
(z_2)$ as follows:
\begin{equation}
  \label{boltztor}
   R_{z_i, z_j} (v_{\alpha} \otimes v_{\beta}) = \sum_{\gamma, \delta}
   R_{\alpha, \beta}^{\gamma, \delta} (z_i, z_j) \; v_{\gamma} \otimes
   v_{\delta} .
\end{equation}

\begin{lemma}
  We have
  \begin{equation}
    \label{r2d2c3p0} (R_{z_j, z_k})_{2 3} (R_{z_i, z_k})_{1 3} (R_{z_i,
    z_j})_{1 2} = (R_{z_i, z_j})_{1 2} (R_{z_i, z_k})_{1 3} (R_{z_j,
    z_k})_{2 3}
  \end{equation}
  as endomorphisms of $V (z_1) \otimes V (z_2) \otimes V (z_3)$.
\end{lemma}
Here the notation is (as usual in quantum group theory) that $X_{i j}$ where
$1 \leqslant i < j \leqslant 3$ means a matrix $X$ acting on the $i,
j$ components in $V (z_1) \otimes V (z_2) \otimes V (z_3)$ with the identity
acting on the third component.

\begin{proof}
  We apply the left-hand side of (\ref{r2d2c3p0}) to
  $v_{\alpha} \otimes v_{\beta} \otimes v_{\gamma}$ and extract the coefficient of $v_{\delta}
  \otimes v_{\varepsilon} \otimes v_{\phi}$. This is found to be
  \[ \sum_{\mu, \nu, \sigma} R_{\alpha, \beta}^{\mu, \sigma} (z_i, z_j)
     R_{\mu, \gamma}^{\delta, \nu} (z_i, z_k) R_{\sigma, \nu}^{\varepsilon,\phi} (z_j, z_k), \]
  which is the partition function of the first system in Theorem~\ref{firstbraid}. The
  same calculation applied to the right hand side of (\ref{r2d2c3p0}) gives
  the partition function of the second system in Theorem~\ref{firstbraid}. So
  they are equal.
\end{proof}

We will now describe two operations that one may perform on the
Boltzmann weights that do not affect the validity of the Yang-Baxter
equation.

\subsection*{Change of basis}
We may change basis in $V (z)$. Let $f (\alpha, z)$ be a function of a
decorated spin $\alpha$ and a complex number $z$. Let $u_{\alpha} = f (\alpha,
z) v_{\alpha}$ for $v_{\alpha} \in V (z)$. Then
\[ R_{z_i, z_j} (u_{\alpha} \otimes u_{\beta}) = \sum_{\gamma, \delta}
   \hat{R}_{\alpha, \beta}^{\gamma, \delta} (z_i, z_j) \; u_{\gamma} \otimes
   u_{\delta} \]
where
\begin{equation}\label{rhatcomp}
\hat{R}_{\alpha, \beta}^{\gamma, \delta} = \frac{f (\alpha, z_i) f (\beta,
   z_j)}{f (\gamma, z_i) f (\delta, z_j)}\,R_{\alpha,\beta}^{\gamma\delta}.
\end{equation}
Note that replacing $R$ by $\hat R$ only affects the weights in cases III, VII, and
VIII in Figure~\ref{comparekojima}. 

Let us translate this into the language of Boltzmann weights.
At the moment we are only concerned with Theorem~\ref{firstbraid}.
Later in Section~\ref{connex} we will apply this technique to the first
Yang-Baxter equation in Theorem~\ref{gammaybe}. Thus we note the effect
on the weights for both types of vertices. Taking the Boltzmann weights
from Figures~\ref{gamgamice} and \ref{gammaice}, with $z_i$ and $z_j$
as in those figures, the weights of
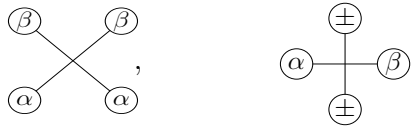
\[\vcenter{\hbox to .7in{\gamgam{\alpha}{\beta}{\beta}{\alpha}{}{}{}{}}},\qquad\qquad
\vcenter{\hbox to .7in{\gammaice{\alpha}{\pm}{\beta}{\pm}{}{}{}{}}}\]
will respectively be multiplied by
\begin{equation}
\label{frecipe}
\frac{f (\alpha, z_i) f (\beta, z_j)}{f (\gamma, z_i) f (\delta, z_j)},\qquad\qquad
\frac{f(\alpha,z_i)}{f(\beta,z_i)}.
\end{equation}
The first statement is a paraphrase of (\ref{rhatcomp}), and the second
is checked the same way.

Returning to the comparison with Kojima's weights, we take 
\[f(\alpha,z)=\left\{\begin{array}{ll}z^a,\qquad& \alpha=+a,\\
1&\alpha=-0. \end{array}\right.
\]
This puts our $R$-matrix into agreement with Kojima in cases III, VII, and
VIII but has no effect on the other cases. 
The modification in this subsection did not fundamentally change
the $R$-matrix, or the quantum group associated to it. We simply made a change of basis in the vector space on which it acts.

\subsection*{\label{twisting}Twisting}
In this subsection we will consider a more fundamental
change of the $R$-matrix which does not affect the validity of the Yang-Baxter equation. This procedure is called
\textit{Drinfeld twisting} \cite{DrinfeldQuasiHopf}. The quantum group associated to the twisted $R$-matrix is not the original quantum group, as the Drinfeld twisting procedure modifies the
comultiplication and universal $R$-matrix of a quasitriangular Hopf algebra. See Chari and Pressley~\cite{ChariPressley} Section~4.2.E for more details. In
Reshetikhin~\cite{ReshetikhinMultiparameter} Section~3, Drinfeld twisting is
used to obtain multiparameter deformations of $U_q(\mathfrak{sl}(n))$.
We explain in~\cite{BBBF}, Section 4 (at least for the $\mathfrak{gl}(n)$ part) how the Drinfeld twist on the quantum group produces the desired change to the $R$-matrix that we present below. 

Notice that in
Figure~\ref{gamgamice}, if we have a nonzero weight for the vertex of
form
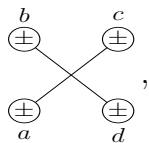
\[\vcenter{\hbox to .7in{\gamgam{\pm}{\pm}{\pm}{\pm}{a}{b}{c}{d}}}\;,\]
then either $a=c$ and $b=d$ or $a=d$ and $b=c$.

Now let us consider a modification of the Boltzmann weights in case~II (i.e., $a=c$ and $b=d$).
We will multiply this weight by a function $\phi(a,b)$ of the decorations $a,b$ that
has the following properties. First, it is independent of $z_i$ and $z_j$.
Second, $\phi(a,b)\,\phi(b,a)=1$.

\begin{proposition}
If $\widetilde{R}$ is the $R$-matrix with this modification of the weights
in case~II, then $\widetilde{R}$ also satisfies the same Yang-Baxter
equation that $R$ does (Theorem~\ref{firstbraid}).
\end{proposition}

\begin{proof}
From the Boltzmann weights in Figure~\ref{comparekojima}, we see
that the decorated spins of the two edges to the right of the vertex will have
the same decorations as the two edges to the left of the vertex, in some order.
From the form of $R_{z_i,z_j}$ it is clear that if either partition
function is nonzero, the decorated spins $\delta$, $\epsilon$ and $\phi$
must be the same as $\alpha$, $\beta$ and $\gamma$ in some order.
From this ordering, we may infer the number of case~II vertices,
and (with an exception to be explained below) it will be the same for both
partition functions. That is, if $+a$ and $+b$ occur on the left in the
opposite order that they do on the right, then a case~II crossing must
occur somewhere on a vertex between the four edges. And this will be
true on both sides of the equation, so multiplying the case~II Boltzmann
weight by $\phi(a,b)$ will have the same effect on both sides of the
equation.

The exception is that if two weights appear in the same order
on the left and right, there may be two case~II vertices or none between them.
Thus suppose that $\alpha=\phi=+a$ and $\beta=\epsilon=+b$. 
Then in the first partition function in Proposition~\ref{secondbraid}
we may have $R_{z_i,z_j}$ and $R_{z_j,z_k}$ either both in case~II or
both in case~III. However if they are both in case~II, the factor that
we have to multiply is $\phi(a,b)\,\phi(b,a)$, which equals 1 by assumption.
\end{proof}

We may use this method of twisting in order to remove the
$g(a-b)$ in case~II, and replace them by $q$, since in
this case $a\not\equiv b$ mod $n$, so $g(a-b)\,g(b-a)=v=q^2$.
We may also adjust the weights in cases V and VI so that in both
cases the coefficient agrees with Kojima's weights.

\subsection*{\label{sign}Sign}
Using the two methods available to us, we see that we can adjust
the Boltzmann weights to agree with Kojima's, up to sign. We must
now discuss the sign. We have agreement for all signs except case~IV.
As Kojima notes (below his equation (2.12)) his $R$-matrix, being
supersymmetric, satisfies a graded Yang-Baxter equation. As he
points out, an ungraded Yang-Baxter equation may be obtained
by changing the sign when all edges are odd-graded. For us,
this would mean changing the sign in cases~I, II and~III. However
it works equally well to change the sign in the case where all
edges are even-graded, that is, in case~IV.

\medskip

In conclusion, putting together the results of all of the above subsections, the supersymmetric Yang-Baxter equation in
Kojima~\cite{KojimaChain} is equivalent to our
Theorem~\ref{firstbraid}.

\section{Intertwining integrals as $R$-matrices\label{connex}}

In this section, we will review results of Kazhdan and
Patterson~\cite{KazhdanPatterson}, Chinta and Offen~\cite{ChintaOffen}
and McNamara~\cite{McNamaraPrincipal,McNamaraCS} concerning
the scattering matrix of the intertwining operators of the
principal series representations on their Whittaker models.
Then we return to the $R$-matrices, using modified Boltzmann
weights that are suited to make a connection with the
notation of \cite{McNamaraCS}, which will be our primary reference.
Finally we will prove Theorem~\ref{commdiagwithintertwiner}.

Let $F$ be a non-archimedean local field with ring of integers $\mathfrak{o}$ and a choice of local uniformizer $\varpi$. Let $q$ be the cardinality of the residue field $\mathfrak{o} / \varpi \mathfrak{o}$. Let $n$ be a fixed positive integer. We assume that $q \equiv 1 \; (\text{mod } 2n)$ so that $F$ contains the $2n$-th roots of unity.
Let $\mu_n$ denote the group of $n$-th roots of unity in $F$ and fix an embedding $\mu_n \longrightarrow \mathbb{C}^\times$.

Let $G:=\GL(r,F)$ and let $T$ be the subgroup of diagonal matrices. We begin by constructing 
a metaplectic $n$-fold cover of $G$, denoted $\tilde{G}^{(n)}$ or just $\tilde{G}$ when the degree of the cover is understood.
Recall that $\tilde{G}$ is constructed as a central extension of $G$ by $\mu_n$:
$$ 1 \longrightarrow \mu_n \longrightarrow \tilde{G} \stackrel{p}{\longrightarrow} G \longrightarrow 1. $$
Thus as a set, $\tilde{G} \simeq G \times \mu_n$, but the multiplication in $\tilde{G}$ is dictated by a choice of cocycle
$\sigma$ for $H^2(G, \mu_n)$. One may construct the cocycle explicitly,
as in Kubota~\cite{Kubota}, Matsumoto \cite{Matsumoto}, Kazhdan and
Patterson~\cite{KazhdanPatterson} and Banks-Levi-Sepanski~\cite{BanksLevySepanski},
or realize the central extension as coming from an extension of $K_2(F)$
constructed by Brylinski-Deligne \cite{BrylinskiDeligne}. For the applications at
hand, we need only a few facts about the multiplication on $\tilde{T} = p^{-1}(T)$,
the inverse image of a maximal split torus $T$ in $G$, and the splitting
properties of some familiar subgroups.

Our cocycle $\sigma$ is chosen so that its restriction to $T\times T\longrightarrow \mu_n$ is given on any $\mathbf{x}, \mathbf{y}$ in $T$ explicitly by
\begin{equation}
  \sigma(\mathbf{x}, \mathbf{y}) = \sigma \left( \left( \begin{array}{ccc} x_1 & & \\ & \ddots & \\  & & x_{r} \end{array} \right),
  \left( \begin{array}{ccc} y_1 & & \\ & \ddots & \\  & & y_{r} \end{array}
  \right) \right) =
  (\det(\mathbf{x}), \det(\mathbf{y}))_{2n} \prod_{i > j} (x_i, y_j)^{-1},
  \label{toruscocycle}
\end{equation}
where $( \cdot, \cdot ) \, : F^\times \times F^\times \longrightarrow \mu_n$ is the
$n$-th power Hilbert symbol and $( \cdot, \cdot )_{2n}$ is the $2n$-th power
Hilbert symbol, so $(x,y)=(x,y)_{2n}^2$.
General properties of the Hilbert symbol may be found in \cite{neukirch} noting
that the symbol there is the inverse of ours; one property we use frequently is
that $(x,x)=1$ for any element $x \in F^\times$, since $F$ contains the $2n$-th
roots of unity.

Let $\Lambda=X_\ast(T)$ denote the group of rational cocharacters of $T$.  The
cocycle $\sigma$ in \eqref{toruscocycle} is the inverse of the one appearing
on p.~39~of \cite{KazhdanPatterson}. A short computation shows that the
commutator of any pair of elements $\tilde{\mathbf{x}}, \tilde{\mathbf{y}}$ in
$\tilde{T}$ projecting to $\mathbf{x}$ and $\mathbf{y}$, respectively, in $T$ is
\begin{equation}
  [\tilde{\mathbf{x}}, \tilde{\mathbf{y}}] = \prod_{i=1}^{r} (x_i, y_i). \label{toruscommutator}
\end{equation}
In particular if $x, y \in F^\times$ and $\lambda, \mu$ are elements of $X_\ast(T)$, let $\widetilde{\lambda(x)}, \widetilde{\mu(y)} \in \tilde{T}$ map to $x^\lambda$ and $y^\mu$, respectively, under the projection $p$ to $T$. Then according to \eqref{toruscommutator},
\begin{equation}
  [ \widetilde{\lambda(x)}, \widetilde{\mu(y)} ] = (x,y)^{\langle \lambda, \mu \rangle}, \label{torusfeltcommutator}
\end{equation}
where $\langle \cdot, \cdot \rangle$ denotes the usual dot product on
$X_\ast(T) \simeq \mathbb{Z}^{r}$. 

In order to make use of results in \cite{McNamaraCS}, we must connect this
explicit construction to the one used there. In \cite{McNamaraCS} the
construction of $\tilde{G}$ is obtained by first constructing the extension of
$G(F)$ by $K_2(F)$ using a $W$-invariant quadratic form $Q$, and then using a
push forward from $K_2(F)$ to the residue field, containing $\mu_n$. The
calculation in \eqref{torusfeltcommutator} implies that the bilinear form
$B(\lambda, \mu) := Q(\lambda+\mu) - Q(\lambda) - Q(\mu)$ for our extension,
as described in Equation~(2.1) of \cite{McNamaraCS}, is given by the dot
product. If $\alpha$ is a (co)root then $Q(\alpha)=1$. 

Finally, we record that the cocycle splits over any unipotent subgroup and
over the maximal compact subgroup $K = \GL(r,\mathfrak{o})$ has a splitting in
$\tilde{G}$. The splitting over the maximal unipotent is clear from the
description of the cocycle in~\cite{KazhdanPatterson} and the splitting
over $K$ is their Proposition~0.1.2. By abuse of notation, we will
denote the image of $K$ in $\tilde{G}$ also as~$K$.

Let $T(\mathfrak{o})=K\cap T$, and let $\tilde{T}(\mathfrak{o})$ be
the preimage of $T(\mathfrak{o})$ in $\widetilde{G}$. Let
$H$ be the centralizer of $\tilde{T}(\mathfrak{o})$ in $\tilde{T}$.
It consists of elements in $\tilde{T}$ whose projection to the torus
$\mathbf{t} = (t_1, \ldots, t_{r}) \in T \simeq (F^\times)^r$ has
$\text{ord}_\varpi \left( t_j \right) \equiv 0 \; (n)$ for $j = 1,\ldots, r$.
The subgroup $H$ is abelian. Thus we may identify
$\tilde{T} / \mu_n \tilde{T}(\mathfrak{o})$ and $H / \mu_n
\tilde{T}(\mathfrak{o})$ with lattices $\Lambda$ and $n \Lambda$,
respectively. In particular $\Lambda$ is isomorphic to the cocharacter lattice
$X_\ast(T)$ of $T$.  The map $\lambda \mapsto \varpi^\lambda$ induces an
isomorphism from $X_\ast(T)$ to $T/T(\mathfrak{o})$. Let
$\mathbf{s}:G\to\tilde G$ denote the standard section. By abuse of notation we
will also denote by $\varpi^\lambda$ the image of $\varpi^\lambda$ under
$\mathbf{s}$. Let $\rho=(r-1,\ldots,2,1,0)$ and let $\Gamma$ be the set of
$\nu\in\Lambda=\mathbb{Z}^r=X_\ast(T)$ such that
\begin{equation}
  \label{prescribedreps}
  \nu - \rho = (c_1,\ldots,c_r),
  \quad \text{with $c_i \in \{0, \ldots, n-1 \}$ for all $i$.} 
\end{equation}
This is a set of coset representatives in $\Lambda=\mathbb{Z}^r=X_\ast(T)$ for
$\Lambda$ modulo $n\Lambda$. Then $\{\varpi^\lambda|\lambda\in\Gamma\}$ are a
set of coset representatives for $\tilde{T} / H$.

Next we recall the construction of the genuine unramified principal series on $\tilde{G}$.
(A representation $\pi$ of $\tilde{G}$ or any subgroup containing $\mu_n$
is called \textit{genuine} if $\pi(\varepsilon g)=\varepsilon\pi(g)$
for $\varepsilon\in\mu_n$, where we are using the fixed embedding of $\mu_n\subset F^\times$
into $\C^\times$.) First we construct geniune irreducible representations
of $\tilde{T}$. Let $\chi$ be a genuine character of $H$ that is trivial
on $\tilde{T}\cap K$; the induced representation $i(\chi)$ of such a character
to $H$ will be irreducible.

Now we may parabolically induce $i(\chi)$ to $\tilde{G}$. This is done by
first inflating the representation from $\tilde{T}$ to $\tilde{B}$, the
inverse image of the standard Borel subgroup $B \supset T$ in $G$ and then
inducing to obtain $I(\chi) :=
\text{Ind}_{\tilde{B}}^{\tilde{G}}(i(\chi))$. Explicitly $I(\chi)$ is the
space of locally constant functions $f : \tilde{G} \longrightarrow i(\chi)$
such that
$$ f(bg) = \delta^{1/2} \chi (b) f(g) \quad \text{for all $g \in \tilde{G}, b \in \tilde{B}$}, $$
where $\delta$ denotes the modular quasicharacter of $B$. Thus $I(\chi)$ is a
$\tilde{G}$-module under the action of right translation.
Let $\phi_K := \phi_K^\chi$ denote any of the $i(\chi)$-valued functions in
the one-dimensional space of $K$-fixed vectors in $I(\chi)$; our results will
be independent of this choice.

The characters $\chi$ of $H$ that are trivial on $\tilde{T}\cap K$
may be parametrized by elements $\mathbf{z}\in\widehat{T}$,
which is the group of diagonal elements of $\GL(r,\C)$. Every
element of $H$ may be written $\varepsilon \mathbf{s}(t)$ with $\varepsilon\in\mu_n$
and $t=\operatorname{diag}(t_1,\cdots,t_r)\in T$ such that
each $\operatorname{ord}(t_i)$ is a multiple of $n$. We
may then define
\[\chi_{\mathbf{z}}(\varepsilon \mathbf{s}(t))=\varepsilon \prod_{i=1}^r z_i^{\operatorname{ord}(t_i)}.\]
We will denote the corresponding principal series representation $\pi_{\mathbf{z}}=I(\chi_{\mathbf{z}})$.
It does not depend uniquely on $\mathbf{z}$ since if $\mathbf{z}^n=(\mathbf{z}')^n$
then~$\pi_{\mathbf{z}}\cong\pi_{\mathbf{z}'}$.

We will assume that $I(\chi_\mathbf{z})$ is irreducible. For this it is
necessary and sufficient to assume that $\mathbf{z}^{n\alpha}\neq q^{\pm1}$
for all roots $\alpha$. This also guarantees that the rational functions that
appear in the sequel do not have poles.

Let $U$ be the subgroup of upper unitriangular matrices in $G$, which is the unipotent
radical of $B$, the positive Borel subgroup. The Matsumoto construction supplies
a splitting of the metaplectic cover over $U$, so by abuse of notation we may
regard $U$ as a subgroup of $\tilde{G}$. In particular, for any positive root
$\alpha \in \Phi^+$, we may regard the one-parameter root subgroup $U_\alpha$ corresponding
to $\alpha$ as a subgroup of $\tilde{G}$.

To any element $w \in W$, the Weyl group, we may define the unipotent subgroup $U_w$ by
$$ U_w := \prod_{\alpha \in \Phi^+, \, w(\alpha) \in \Phi^-} U_\alpha. $$
Then define the intertwining operator $\mathcal{A}_w : I(\chi) \rightarrow I( {}^w \chi)$ by
\begin{equation} \mathcal{A}_w (f) (g) := \int_{U_w} f(w^{-1} u g) \, du \label{intertwinerunnorm} \end{equation}
whenever the above integral is absolutely convergent, and by the usual meromorphic continuation in general. 
(By abuse of notation we are using the same letter $w$ for the Weyl group and for a representative
in~$K$.)

The representation $\pi_{\mathbf{z}}$ contains a $K$-fixed vector
$\phi_K$, unique up to constant multiple. We may choose these so that
$$ \mathcal{A}_w \phi_K^{\mathbf{z}} = c_w(\chi) \phi_K^{w\mathbf{z}} $$
where for any simple reflections $s = s_\alpha$ and any $w$ such that the length function $\ell(s_\alpha w) = \ell(w) + 1$,
$$ c_s(\chi) = \frac{1 - q^{-1} \mathbf{z}^{n \alpha}}{1 - \mathbf{z}^{n \alpha}}, \quad \text{and} \quad c_{sw}(\chi) = c_s(\chi^w) c_w(\chi). $$
Let $\bar{\mathcal{A}}_w$ denote the normalized intertwiner:
\begin{equation} \bar{\mathcal{A}_w} := c_w(\chi)^{-1} \mathcal{A}_w. \label{normalizedint} \end{equation}

Let $\psi$ be a character of $U$ such that if $i_\alpha$ is the embedding $\SL_2\to \GL_n$ along
the simple root $\alpha$, then the additive character
$x\mapsto i_\alpha\!\left(\begin{smallmatrix}1&x\\&1\end{smallmatrix}\right)$
of $F$ is trivial on $\mathfrak{o}$ but no larger fractional ideal.
A \textit{Whittaker functional} on a representation $(\pi, V)$ of $\tilde{G}$
is a linear functional $W^{\pi}$ for which
\[ W^{\pi}  (\pi (u) v) = \psi (u) W^{\pi} (v)  \quad \text{for all $u \in U$
   and } v \in V. \]
As stated in Section~6 of {\cite{McNamaraCS}}, the dimension of the space of
Whittaker functionals for the principal series $I (\chi)$ is equal to the
cardinality $n^r$ of $\tilde{T} / H$. Let $\mathbf{W}^{\chi}$ denote the $i
(\chi)$-valued Whittaker functional on $I (\chi)$ defined by
\begin{equation}
  \label{ichiwhittaker}
  \mathbf{W}^{\chi} (\phi) := \int_{U^-} \phi (uw_0) \overline{\psi (u)}\, du
  \; : I (\chi) \longrightarrow i (\chi) .
\end{equation}
(We denote this $\mathbf{W}^{\mathbf{z}}$ when $\chi=\chi_{\mathbf{z}}$.)
Then there is an isomorphism between the linear dual $i (\chi)^{\ast}$ and the
space Whittaker functionals to $\mathbb{C}$ on $I (\chi)$ given by
\begin{equation}
  \mathcal{L} \longmapsto \mathcal{L} \circ W^{\chi}, \quad \text{for
  $\mathcal{L}$ in $i (\chi)^{\ast}$.} \label{functionalcorresp}
\end{equation}

Let us describe a particular basis of $i (\chi)^{\ast}$ used in
{\cite{McNamaraCS}} for the computation of the spherical function under the
Whittaker functional. Let $v_0 := \phi_K (1)$, an element of $i (\chi)$.
Let $\theta_\chi$ denote the representation of $\tilde{T}$ on $i (\chi)$
(denoted $\pi_\chi$ in~\cite{McNamaraCS}). Then
$\{ \theta_{\chi} (\varpi^{\gamma})v_0 | \gamma \in \Gamma \}$ is a basis for
$i (\chi)$. Let $\{ \mathcal{L}_{\gamma}^{(\chi)} \}$ denote
the dual basis of $i (\chi)^{\ast}$. If $\mu \in \Lambda$ write $\mu = \beta +
\gamma$ with $\gamma \in \Gamma$ and $\beta \in n \Lambda$. Then
\begin{equation}
  \mathcal{L}_{\nu}^{(\chi)}  (\theta_{\chi} (\varpi^{\mu}) v_0) =
  \left\{ \begin{array}{ll}
    \chi(\varpi^{\beta}) & \text{if $\nu = \gamma$,}\\
    0 & \text{otherwise.}
  \end{array} \right. \label{lambdafunctional}
\end{equation}
Thus we obtain a basis of the space of Whittaker functionals on $I (\chi)$,
denoted $W_{\gamma}^{\chi} =\mathcal{L}_{\gamma}^{\chi} \circ \mathbf{W}^{\chi}$ using
the isomorphism (\ref{functionalcorresp}). We will denote
$W_{\gamma}^{\mathbf{z}} = W_{\gamma}^{\chi}$ if $\chi =\chi_{\mathbf{z}}$,
or as simply~$W_\gamma$.

The spherical Whittaker function $W_{\gamma}^{\mathbf{z}}  (\pi
(\varpi^{\lambda}) \phi_K)$ vanishes unless the weight $\lambda$ is dominant.
One approach to studying them, going back to Casselman and Shalika (for linear
groups) and Kazhdan and Patterson for metaplectic covers, is to exploit the
fact that $\mathcal{W}^{^w \chi} \circ \overline{\mathcal{A}_w}$ is an $i
(\chi)$-valued Whittaker functional for $I (\chi)$. This is the approach that
was taken by Chinta and Offen {\cite{ChintaOffen}} and McNamara
{\cite{McNamaraCS}}. Thus we expand
\begin{equation}
  \label{whitequality}
  W_{\mu}^{w\mathbf{z}} \circ \overline{\mathcal{A}}_w = \sum_{\nu \in
    \Gamma} \tau_{\mu, \nu} W_{\nu}^{\mathbf{z}}
\end{equation}
for some rational functions
$\tau_{\mu, \nu} = \tau^{(w)}_{\mu, \nu}(\mathbf{z}^n)$.
It suffices to understand these structure constants on simple
reflections $w = s_{\alpha}$. These were computed for metaplectic covers of
$\GL (r)$ by Kazhdan and Patterson, and we discuss their calculation
following Theorem 13.1 in {\cite{McNamaraCS}}.

We now introduce the Gauss sums $g(a)$, which depend on $a$ modulo $n$ and satisfy
the conditions $g(0)=-v$, while $g(a)\,g(n-a)=v$ if $n$ does not divide $a$, with
$v=q^{-1}$. These are given by the formula
\begin{equation}
  \label{gausssumdef}
  g(a)=\frac{1}{q}\sum_{t \in(\mathfrak{o}/(\varpi))^\times}
  ({\varpi},t)^a\psi\!\left(\frac{t}{\varpi}\right).
\end{equation}

\begin{proposition}[Kazhdan-Patterson, \cite{KazhdanPatterson}, Lemma I.3.3] \label{KPlemma}
Let $s=s_{\alpha}$ be a simple reflection and let $\mu,\nu\in\Gamma$.
The structure constants $\tau_{\nu,\mu} := \tau_{\nu, \mu}^{(w)}$ for 
can be broken into two pieces:
\[\tau_{\nu,\mu} = \tau^1_{\nu,\mu} + \tau^2_{\nu,\mu}\]
where $\tau^1$ vanishes unless $\nu \sim \mu$ mod $n\Lambda$ and $\tau^2$ vanishes unless $\nu \sim s(\mu) + \alpha$ mod $n\Lambda$. Moreover:
\begin{equation}
  \label{tauone}
  \tau^1_{\mu, \mu} =
  (1 - q^{-1}) \frac{\mathbf{z}^{n \lceil \frac{\langle \alpha, \mu
        \rangle}{n} \rceil \alpha}}{1-q^{-1} \mathbf{z}^{n \alpha}}
\end{equation}
where $\lceil x \rceil$ denotes the smallest integer at least $x$, and
\begin{equation}
  \label{tautwo}
  \tau^2_{s(\mu)+\alpha, \mu} =
  g(\langle \alpha, \mu - \rho \rangle) \frac{1-\mathbf{z}^{n
      \alpha}}{1-q^{-1} \mathbf{z}^{n \alpha}}.
\end{equation}
\end{proposition}

\begin{proof} This is Theorem~13.1 in~\cite{McNamaraCS}.
Recall that $Q(\alpha) = 1$ on simple roots $\alpha$ so
$n_\alpha = n / \gcd(n, Q(\alpha)) = n$. Our cocycle has been chosen so that
$B(\alpha, \mu) = \langle \alpha, \mu \rangle$. Our $n$-th order Gauss sum $g$ is
$q^{-1}\mathfrak{g}$ in the notation of \cite{McNamaraCS}. Finally, we have
$x_\alpha =\mathbf{z}^{\alpha}$ to obtain (\ref{tauone}) and~(\ref{tautwo}).
\end{proof}


We now return to the $R$-matrices. We will not be concerned with
partition functions in this section but we will use the notation
of Boltzmann weights (slightly modified) in order to make a connection with
Proposition~\ref{KPlemma}.

The weights we need now are given in Figures~\ref{modifiedweights} and~\ref{modifiedrwt};
we will derive these from those in Figures~\ref{gammaice} and~\ref{gamgamice}.
We remind the reader that in Section~\ref{partitionfunction}, we stressed
that it suffices to consider only the decoration $0$ associated to a $-$ spin,
and the table below reflects this assumption. So we omit other $-a$ decorated
spins with $a\neq0$.

\begin{proposition}
The Yang-Baxter equation is satisfied with the weights
in Figures~\ref{modifiedweights} and~\ref{modifiedrwt}.
\end{proposition}

\begin{proof}
To obtain these from the weights in Figure~\ref{gammaice} and~\ref{gamgamice},
we make use of the \textit{change of basis} method described in
Section~\ref{supersection}. We take the function $f(\alpha,z)$ to equal
\[\left\{\begin{array}{ll}z^a&\text{if $\alpha=+a$, $0\leqslant a<n$,}\\
1&\text{if $\alpha=-0$.}\end{array}\right.\]
In Figure~\ref{modifiedweights} we further divide each weight by
$z_i$, and in Figure~\ref{modifiedrwt} we divide by $z_1^n-vz_2^n$.
These multiplications apply to all weights, so we may do this at our
convenience without affecting the Yang-Baxter equation.
\end{proof}

\begin{figure}[h]
\[
\begin{array}{|c|c|c|c|c|c|}
\hline
\tt{a}_1&\tt{a}_2&\tt{b}_1&\tt{b}_2&\tt{c}_1&\tt{c}_2\\
\hline
\begin{array}{c}\gammaice{+}{+}{+}{+}{a+1}{a}\\z_i^{-n\delta(a+1)}\end{array} &
\begin{array}{c}\gammaice{-}{-}{-}{-}{0}{0}\\\ 1 \end{array} &
\begin{array}{c}\gammaice{+}{-}{+}{-}{a+1}{a}\\g(a)z_i^{-n\delta(a+1)}\end{array} &
\begin{array}{c}\gammaice{-}{+}{-}{+}{0}{0}\\ 1 \end{array} &
\begin{array}{c}\gammaice{-}{+}{+}{-}{0}{0}\\ 1-v \end{array} &
\begin{array}{c}\gammaice{+}{-}{-}{+}{1}{0}\\z_i^{-n\delta(1)} \end{array} \\
\hline
\end{array}
\]
\caption{Modified Boltzmann weights.
  \label{modifiedweights}}
\end{figure}

\begin{figure}[h]
\[ {\tabulinesep=1mm \begin{tabu}{|c|c|c|c|c|c|}
\hline \tt{a}_1:
\begin{array}{c}\gamgam{+}{+}{+}{+}{a}{a}{a}{a}\\ \displaystyle\frac{-v + \mathbf{z}^{n \alpha}}{1-v \mathbf{z}^{n \alpha}} \end{array} & 
\begin{array}{c}\gamgam{+}{+}{+}{+}{b}{a}{b}{a}\\ g(a-b) \displaystyle\frac{1- \mathbf{z}^{n \alpha}}{1-v \mathbf{z}^{n \alpha}} \end{array} & 
\begin{array}{c}\gamgam{+}{+}{+}{+}{b}{a}{a}{b}\\
\frac{(1-v)}{1-v \mathbf{z}^{n \alpha}} \cdot \left\{\begin{array}{ll}
\scriptstyle \mathbf{z}^{n \alpha}\kern-5pt &\scriptstyle a>b, \\
\scriptstyle 1 &\scriptstyle a<b \end{array}\right.
\end{array} & 
\tt{a}_2: 
\begin{array}{c}\gamgam{-}{-}{-}{-}{0}{0}{0}{0}\\1\end{array} \\ 
\hline
\tt{b}_1:\begin{array}{c}\gamgam{+}{-}{+}{-}{a}{0}{a}{0}\\ \displaystyle\frac{v(1 - \mathbf{z}^{n \alpha})}{1-v \mathbf{z}^{n \alpha}} \end{array} & 
\tt{b}_2:\quad\begin{array}{c}\gamgam{-}{+}{-}{+}{0}{a}{0}{a}\\ \displaystyle\frac{1 - \mathbf{z}^{n \alpha}}{1-v \mathbf{z}^{n \alpha}} \end{array} & 
\tt{c}_1:\begin{array}{c}\gamgam{-}{+}{+}{-}{0}{a}{a}{0}\\ \displaystyle\frac{(1-v) \mathbf{z}^{n \alpha}}{1-v \mathbf{z}^{n \alpha}} \end{array} & 
\tt{c}_2:\begin{array}{c}\gamgam{+}{-}{-}{+}{a}{0}{0}{a}\\ \displaystyle\frac{(1-v)}{1-v \mathbf{z}^{n \alpha}} \end{array} \\ 
\hline
\end{tabu} } \] 
\caption{Modified weights $\hat{R}_{\mathbf{z}}$. When combined with the
  weights in Figure~\ref{modifiedweights}, they satisfy a Yang-Baxter
  equation. This follows from Theorem~\ref{gammaybe}. In this figure
  we are assuming that charges depicted as $a$ and $b$ are in distinct residue classes mod~$n$.
  \label{modifiedrwt}}
\end{figure}

\begin{proposition} \label{taumatch}
Let $\mu \in X_\ast(T) \simeq \mathbb{C}[\Lambda]$ with $\mu - \rho = (c_1,\cdots,c_r)$ for some integers $c_i \in [0,n)$. Let $\tau_{\nu, \mu}(\mathbf{z}) := \tau^{(s_i)}_{\nu,\mu}(\mathbf{z})$ as in
Proposition~\ref{KPlemma}.
Let $\operatorname{wt}$ be the weights for $\hat{R}$ in Figure~\ref{modifiedrwt} above with $v= q^{-1}$. Given any pair of integers $a,b$ with $a \equiv c_i$ and $b \equiv c_{i+1}$ mod $n$, if $a \not\equiv b$ mod $n$, then
$$ \tau^1_{\mu, \mu}(\mathbf{z}) = \operatorname{wt} \left( \begin{array}{c} \gamgam{+}{+}{+}{+}{a}{b}{b}{a} \end{array} \right) \quad \text{and} \quad \tau^2_{s_i(\mu)+\alpha_i, \mu}(\mathbf{z}) = \operatorname{wt} \left( \begin{array}{c}\gamgam{+}{+}{+}{+}{b}{a}{b}{a} \end{array} \right). $$
If $a \equiv b$ mod $n$, then both $\tau^1_{\mu, \mu}(\mathbf{z})$ and $\tau^2_{s_i(\mu)+\alpha_i, \mu}(\mathbf{z})$ are nonzero and
\begin{equation}
  \label{tauequalcase}
  \tau^1_{\mu, \mu}(\mathbf{z}) + \tau^2_{s_i(\mu)+\alpha_i, \mu}(\mathbf{z}) = \operatorname{wt} \left( \begin{array}{c}\gamgam{+}{+}{+}{+}{a}{a}{a}{a} \end{array} \right).
\end{equation}
\end{proposition}

\begin{proof}
We begin by rewriting $\tau^1$ and $\tau^2$ in terms of $c_i$ and $c_{i+1}$. Recall from (\ref{tauone}) that with $\alpha = \alpha_i$
\begin{align*}
\tau^1_{\mu, \mu} = ( (1 - q^{-1}) \frac{\mathbf{z}^{n \lceil \frac{\langle
      \alpha, \mu \rangle}{n} \rceil \alpha}}{1-q^{-1} \mathbf{z}^{n \alpha}}
=& (1 - q^{-1}) \frac{\mathbf{z}^{n \lceil \frac{c_i - c_{i+1}+1}{n} \rceil \alpha}}{1-q^{-1} \mathbf{z}^{n \alpha}} \\=& \frac{(1 - q^{-1})}{1-q^{-1} \mathbf{z}^{n \alpha}} \begin{cases}
  \mathbf{z}^{n \alpha} & \text{if $c_i - c_{i+1} \geqslant 0$} \\
  1 & \text{if $c_i - c_{i+1} < 0.$} \end{cases}
\end{align*}
where the second equality used that $\langle \alpha, \mu-\rho \rangle = c_i-c_{i+1}$. From (\ref{tautwo}),
$$ \tau^2_{s(\mu)+\alpha, \mu} = g(\langle \alpha, \mu - \rho \rangle) \frac{1-\mathbf{z}^{n \alpha}}{1-q^{-1} \mathbf{z}^{n \alpha}} = g(c_i-c_{i+1}) \frac{1-\mathbf{z}^{n \alpha}}{1-q^{-1} \mathbf{z}^{n \alpha}}. $$
Note that $\mu = s_i(\mu) + \alpha_i$ if and only if $c_i = c_{i+1}$, since $s_i(\mu) + \alpha_i = s_i(\mu - \rho) + \rho$, so $\tau^1$ and $\tau^2$ are only simultaneously nonzero when $c_i = c_{i+1}$.

Now we compare to the weights in Figure~\ref{modifiedrwt} with $a-b \equiv c_{i+1} - c_i$ (mod $n$) according to cases. First if $c_i \ne c_{i+1}$ so that $a \not\equiv b$ (mod $n$), then the modified $R$-vertex weight in the top row, third column entry of Figure~\ref{modifiedrwt} indeed matches the evaluation of $\tau^1$ above upon setting $v=q^{-1}$. Moreover, the top row, second column entry of Figure~\ref{modifiedrwt} agrees with $\tau^2$ under the same specialization $v = q^{-1}$.

To finish, consider the case when $c_i = c_{i+1}$ so that both $\tau^1$ and $\tau^2$ are nonzero. Then
$$ \tau^1_{\mu, \mu} + \tau^2_{s(\mu)+\alpha, \mu} = \frac{(1 - q^{-1}) \mathbf{z}^{n \alpha}}{1-q^{-1} \mathbf{z}^{n \alpha}} -q^{-1} \frac{1-\mathbf{z}^{n \alpha}}{1-q^{-1} \mathbf{z}^{n \alpha}} = \frac{\mathbf{z}^{n \alpha} - q^{-1}}{1-q^{-1} \mathbf{z}^{n \alpha}} $$
since $g(0) = -v =-q^{-1}$. And this in turn is precisely the weight of the $R$-vertex in the top row, first column of Figure~\ref{modifiedrwt} where $a \equiv c_i = c_{i+1}$ and $v=q^{-1}$.
\end{proof}

Recall that the module of Whittaker coinvariants $\pi_{\mathbf{z}, \psi}$
is naturally the dual space of the space $\mathcal{W}^{\mathbf{z}}$ of
Whittaker functionals on $\pi_{\mathbf{z}}$. If $\gamma\in\Gamma$, let
$\Omega_\mu$ be the image of $\mathbf{s}(\varpi^\mu)$ in $\pi_{\mathbf{z}, \psi}$.
Using (\ref{lambdafunctional}) this $\{\Omega_\mu\}$ is the basis of
$\pi_{\mathbf{z}, \psi}$ dual to the basis $W^{\mathbf{z}}_\mu$ of
$\mathcal{W}^{\mathbf{z}}$. Then the map $\overline{\mathcal{A}}_{s_i}$
induces the map
\begin{equation}
  \label{whiteqdual} \overline{\mathcal{A}}_{s_i} (\Omega_{\mu}) = \sum_\nu \tau_{\nu,\mu} \Omega_\nu .
\end{equation}
This follows from (\ref{whitequality}) by duality. We define the map
$\theta_{\mathbf{z}} : \pi_{\mathbf{z}, \psi} \longrightarrow
\bigoplus_i V_+(z_i)$ (needed for Theorem~\ref{commdiagwithintertwiner}) by
\[ \theta_{\mathbf{z}} (\Omega_{\mu}) = v_{c_1} \otimes \cdots \otimes
   v_{c_r}  \]
when $\mu \in \Gamma$. Recall that this means $\mu - \rho = (c_1, \cdots, c_r)$
with $0\leqslant c_i <n$. We will use the notation $v_{\mu - \rho}$ to denote this vector.

We are now ready to prove one of our main results.

\begin{proof}[Proof of Theorem~\ref{commdiagwithintertwiner}]
  Let $\mu \in \Gamma$. Let $\nu = s_i \mu + \alpha$. Write $\mu - \rho =
  (c_1, \cdots, c_r)$ and $\nu = s_i \mu + \alpha$, so that $\nu - \rho = s_i
  (\mu - \rho)$ has the same components with $c_i$ and $c_{i + 1}$
  interchanged.
  
  We consider the case where $\mu \neq \nu$. We have
  \[ \theta_{s_i \mathbf{z}} \overline{\mathcal{A}}_{s_i} (\Omega_{\mu}) =
     \theta (\tau_{\mu, \mu} \Omega_{\mu} + \tau_{\nu, \mu} \Omega_{\nu}) =
     \tau_{\mu, \mu} v_{\mu - \rho} + \tau_{\nu, \mu} v_{\nu - \rho} . \]
  On the other hand using (\ref{boltztor})
  \[ \tau R (v_{c_i} \otimes v_{c_j}) = \sum_{c_k, c_l} R^{c_k, c_l}_{c_i,
     c_j} (v_{c_l} \otimes v_{c_k}) . \]
  Taking $j=i+1$, on the right-hand side, the only nonzero terms are $(c_l, c_k) = (c_i, c_{i+1})$
  or $(c_{i+1}, c_i)$. So
  \[ (\tau R)_{i, i + 1} \theta_{\mathbf{z}} (\Omega_{\mu}) = (R_{c_i, c_{i
     + 1}}^{c_{i + 1}, c_i})_{i, i + 1} v_{\mu - \rho} + (R_{c_i, c_{i +
     1}}^{c_i, c_{i + 1}})_{i, i + 1} v_{\nu - \rho} . \]
  (The subscript $X_{i, i + 1}$ means that the operator is applied in the
  $i, i + 1$ position of the $r$-fold tensor product $V_{+\mathbf{z}}$.) Thus we
  need
  \[ (R_{c_i, c_{i + 1}}^{c_{i + 1}, c_i})_{i, i + 1} = \tau_{\mu, \mu},
     \qquad (R_{c_i, c_{i + 1}}^{c_i, c_{i + 1}})_{i, i + 1} = \tau_{\nu, \mu}
  \]
  and this is the content of Proposition~\ref{taumatch}.

  The case where $\mu=\nu$ is similar, using (\ref{tauequalcase}).
\end{proof}

\section{Functional Equations via Partition Functions}

In Theorem~\ref{matchwhit}, we gave one spherical Whittaker
function as the partition function of a solvable lattice model. However there
are $n^r$ independent Whittaker functions. In this section, we will show that
the charge statistic can be refined to give $n^r$ independent Whittaker
functions. In this section we will use the unmodified weights in
Figure~\ref{gammaice}.

In Theorem~\ref{matchwhit}, we were vague as to the precise
normalization. Now that we have defined enough notation in the previous
section, let us give the precise normalization. Let
$\chi = \chi_{\mathbf{z}}$. There is a unique functional $\mathcal{L}^{\circ}$
on $i(\chi)$ such that
\begin{equation}
  \mathcal{L}^{\circ} (\theta (\varpi^{\lambda}) v_0) =\mathbf{z}^{\lambda}\;.
\end{equation}
(This was denoted $\lambda$ in {\cite{McNamaraDuke}}, which underlies the
proof of Theorem~\ref{commdiagwithintertwiner}.) Evidently
\[ \mathcal{L}^{\circ} = \sum_{\gamma \in \Gamma} \mathbf{z}^{\gamma}
   \mathcal{L}_{\gamma}, \]
where $\Gamma$ is the set of representatives defined in
(\ref{prescribedreps}). We define, for $g \in \tilde{G}$
\[ W^{\circ} (g) =\mathcal{L}^{\circ} \mathbf{W} (\pi (g) \phi_K) , \]
where $\mathbf{W}$ is defined in (\ref{ichiwhittaker}).
The correct normalization for Theorem~\ref{matchwhit} is:
\begin{equation}
  \label{thmoneprec}
  Z (\mathfrak{S}_{\lambda}) =\mathbf{z}^{w_0 \rho}
  \delta^{- 1 / 2} (\varpi^{\lambda}) W^{\circ} (\varpi^{\lambda}) .
\end{equation}
Now if $\gamma \in \Gamma$ then
\[ W_{\gamma} (g) = \mathcal{L}_{\gamma} W (\pi (g)
   \phi_K), \]
so that $W^{\circ} = \sum \mathbf{z}^{\gamma} W_{\gamma}$.

Let $\mathfrak{M}=\mathbb{C} (z_1, \cdots, z_r, v)$ be the field of rational
functions in $z_i$ and $v$, which we may think of as indeterminates. Let
$\mathfrak{M}_n =\mathbb{C} (z_1^n, \cdots, z_r^n, v)$. Then $[\mathfrak{M}:
\mathfrak{M}_n] = n^r$ and a basis of $\mathfrak{M}$ over $\mathfrak{M}_n$
consists of the $n^r$ elements $\mathbf{z}^{\gamma}$ where $\gamma \in
\Gamma$, the set of representatives defined by (\ref{prescribedreps}).

\begin{lemma}
  \label{lpilemma}Let $g \in \tilde{G}$ and let $\gamma \in \Gamma$. Then
  $\mathcal{L}_{\gamma} (\pi (g) \phi_K)$ is in $\mathfrak{M}_n$.
\end{lemma}

\begin{proof}
  We make an Iwasawa decomposition $g = \varepsilon t v \varpi^{\nu} k$ where
  $\varepsilon \in \mu_n$, $t \in T (\mathfrak{o})$, $u \in U$, $\nu \in
  \Lambda$ and $k \in K$. Then $\mathcal{L}_{\gamma} (\pi (g) \phi_K) =
  \varepsilon \mathcal{L} (\pi (\varpi^{\nu}) \phi_K) = \varepsilon
  \mathcal{L}_{\gamma} (\theta (\varpi^{\lambda}) v_0)$. This depends only
  $\mathbf{z}^n$ by (\ref{lambdafunctional}).
\end{proof}

\begin{proposition}
  \label{whitinmn}Let $\lambda$ be a dominant weight, and let $\gamma \in
  \Gamma$. Then the function
  $\mathbf{z}^{\gamma} W_{\gamma}^{\mathbf{z}} (\varpi^{\lambda}\phi_K)$,
  considered as a function of $\mathbf{z}$, lies in the coset
  $\mathbf{z}^{\gamma} \mathfrak{M}_n$.
\end{proposition}

\begin{proof}
  The function $W_{\gamma}^{\mathbf{z}} (\varpi^{\lambda} \phi_K)$ is
  defined by the integral
  \[ \int_{U_-} \mathcal{L}_{\gamma}^{\mathbf{z}}
     \phi_K (u w_0 \varpi^{\lambda}) \overline{\psi (u)} \, d u. \]
  For every value of $u$ the integrand is in $\mathfrak{M}_n$ by
  Lemma~\ref{lpilemma}. Hence the integral is, also.
\end{proof}

If $x \in \mathbb{Z}$ let $[x]$ denote the least residue of $x$ modulo $n$.
The following result is a refinement of Theorem~\ref{matchwhit}.

\begin{theorem}
  \label{whitpartcharge}
  Let $(c_1, \cdots, c_r) \in \mathbb{Z}^r$ and let $\gamma \in \Gamma$ be
  defined by $\gamma_i - r + i = [N + 1 - r - c_i]$. Then
  \[ Z (\mathfrak{S}_{\lambda} ; \mathbf{c}) =\mathbf{z}^{w_0 \rho+\gamma}
     \delta^{- 1 / 2} (\varpi^{\lambda}) W_{\gamma} (\varpi^{\lambda}) . \]
\end{theorem}

\begin{proof}
  Since as a vector space $\mathfrak{M}= \bigoplus_{\gamma \in \Gamma}
  \mathbf{z}^{\gamma} \mathfrak{M}_n$, we may project both sides in
  (\ref{thmoneprec}) onto $\mathbf{z}^{\gamma + w_0 \rho} \mathfrak{M}_n$.
  It follows from Proposition~\ref{whitinmn} that the projection of the
  right-hand side is $\mathbf{z}^{\gamma + w_0 \rho} \delta^{- 1 / 2}
  (\varpi^{\lambda}) W_{\gamma} (\varpi^{\lambda})$. As for the left-hand side
  consider a state $\mathfrak{s}$ of the system $\mathfrak{S}_{\lambda}$. We
  observe that in the Boltzmann weights in Figure~\ref{gammaice} the vertex
  contributes a $z_i$ to the partition function if the spin to the left of the
  vertex is $-$, while it increments the charge if the spin to the left is
  $+$. Thus let $c_i$ be the charge at the left edge in the $i$-th row, and
  let $\nu_i$ be the power of $z_i$ that appears in the Boltzmann weight of
  the state. We see that $c_i + \nu_i = N$. Now if the Boltzmann weight of the
  state is in $\mathbf{z}^{\gamma + w_0 \rho} \mathfrak{M}_n$ we must have
  $\nu_i \equiv \gamma_i + i - 1$ modulo $N$. Now $\gamma_i - \rho_i =
  \gamma_i - r + i \equiv \nu_i - r + 1 \equiv N + 1 - r - c_i$. Remembering
  that $\gamma \in \Gamma$ means that $0 \leqslant \gamma_i - \rho_i < n$, we
  see that the state $\mathfrak{s}$ contributes to $Z (\mathfrak{S}_{\lambda}
  ; \mathbf{c})$ if and only if $\gamma_i - r + i = [N + 1 - r - c_i]$ and
  the statement follows.
\end{proof}

Our next result is a variant of Theorem~\ref{commdiagwithintertwiner}
that describes the functional equations of the partition function.
In view of \eqref{thmoneprec}, this can also be regarded as a functional
equation for the Whittaker functions. We will add $\mathbf{z}$ to the notation
and denote
$Z(\mathfrak{S}_\lambda;\mathbf{c})=Z(\mathfrak{S}_{\lambda,\mathbf{z}};\mathbf{c})$.
Let $s_i$ be the simple reflection that interchanges $i$ and $i+1$. In our
notation, note that $\mathbf{z}^{\alpha_i}=z_i/z_{i+1}$, where
$\alpha_i$ is the $i$-th simple root.

\begin{proposition}
  \label{equivmetaplecticice}
  Let $c$ be the least residue of $c_i-c_{i+1}$ modulo $n$. Then
  \begin{equation}
    \label{partfe}
  Z(\mathfrak{S}_{\lambda,s_i\mathbf{z}};s_i\mathbf{c})=
  (1-v)\frac{\mathbf{z}^{(n-c)\alpha_i}}{1-v\mathbf{z}^{n\alpha}}
  Z(\mathfrak{S}_{\lambda,\mathbf{z}};s_i\mathbf{c})
  +g(c_i-c_{i+1})\frac{1-\mathbf{z}^{n\alpha}}
  {1-v\mathbf{z}^{n\alpha}}
  Z(\mathfrak{S}_{\lambda,\mathbf{z}};\mathbf{c}).
  \end{equation}
\end{proposition}

\begin{proof}
  We attach the vertex $R_{z_i,z_{i+1}}$ to the right of the
  partition function of the system
  $(\mathfrak{S}_{\lambda,s_i\mathbf{z}};s_i\mathbf{c})$,
  thus:
  \[\scalebox{.85}
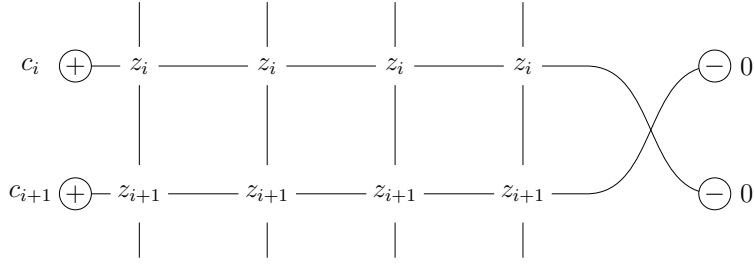
{\begin{tikzpicture}
  \draw (0,0)--(8,0);
  \draw (0,2)--(8,2);
  \draw (1,-1)--(1,3);
  \draw (3,-1)--(3,3);
  \draw (5,-1)--(5,3);
  \draw (7,-1)--(7,3);
  \coordinate (a1) at (8,0);
  \coordinate (c1) at (10,2);
  \coordinate (a2) at (8,2);
  \coordinate (c2) at (10,0);
  \draw (a1) to [out=0,in=180] (c1);
  \draw (a2) to [out=0,in=180] (c2);
  \draw[fill=white] (0,2) circle (.25);
  \draw[fill=white] (0,0) circle (.25);
  \draw[fill=white] (10,2) circle (.25);
  \draw[fill=white] (10,0) circle (.25);
  \path[fill=white] (1,0) circle (.45);
  \path[fill=white] (1,2) circle (.3);
  \path[fill=white] (3,0) circle (.45);
  \path[fill=white] (3,2) circle (.3);
  \path[fill=white] (5,0) circle (.45);
  \path[fill=white] (5,2) circle (.3);
  \path[fill=white] (7,0) circle (.45);
  \path[fill=white] (7,2) circle (.3);
  \node at (1,2) {$z_i$};
  \node at (1,0) {$z_{i+1}$};
  \node at (3,2) {$z_i$};
  \node at (3,0) {$z_{i+1}$};
  \node at (5,2) {$z_i$};
  \node at (5,0) {$z_{i+1}$};
  \node at (7,2) {$z_i$};
  \node at (7,0) {$z_{i+1}$};
  \node at (0,0){$+$};
  \node at (0,2){$+$};
  \node at (10,0){$-$};
  \node at (10,2){$-$};
  \node at (10.5,0){$0$};
  \node at (10.5,2){$0$};
  \node at (-.7,0){$c_{i+1}$};
  \node at (-.7,2){$c_i$};
  \end{tikzpicture}}
  \]
  Consulting Figure~\ref{gamgamice}, there is only one possible configuration
  for the $R$-vertex, so attaching it just multiplies the partition function
  by $z_{i+1}^n-vz_i^n$. Now using the Yang-Baxter equation we may move the
  $R$-matrix to the left, and obtain the partition function of a system that
  looks like this:
  \[\scalebox{.85}
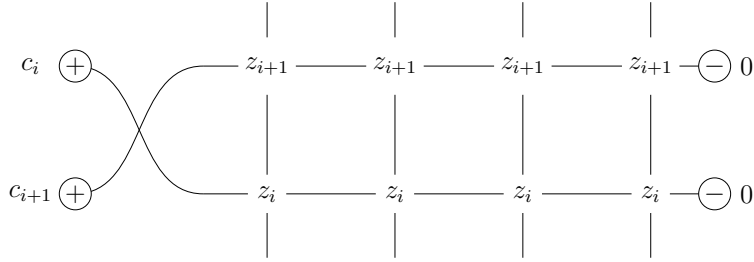
{\begin{tikzpicture}
  \draw (0,0)--(8,0);
  \draw (0,2)--(8,2);
  \draw (1,-1)--(1,3);
  \draw (3,-1)--(3,3);
  \draw (5,-1)--(5,3);
  \draw (7,-1)--(7,3);
  \coordinate (a1) at (-2,0);
  \coordinate (c1) at (0,2);
  \coordinate (a2) at (-2,2);
  \coordinate (c2) at (0,0);
  \draw (a1) to [out=0,in=180] (c1);
  \draw (a2) to [out=0,in=180] (c2);
  \draw[fill=white] (-2,2) circle (.25);
  \draw[fill=white] (-2,0) circle (.25);
  \draw[fill=white] (8,2) circle (.25);
  \draw[fill=white] (8,0) circle (.25);
  \path[fill=white] (1,0) circle (.3);
  \path[fill=white] (1,2) circle (.45);
  \path[fill=white] (3,0) circle (.3);
  \path[fill=white] (3,2) circle (.45);
  \path[fill=white] (5,0) circle (.3);
  \path[fill=white] (5,2) circle (.45);
  \path[fill=white] (7,0) circle (.3);
  \path[fill=white] (7,2) circle (.45);
  \node at (1,0) {$z_i$};
  \node at (1,2) {$z_{i+1}$};
  \node at (3,0) {$z_i$};
  \node at (3,2) {$z_{i+1}$};
  \node at (5,0) {$z_i$};
  \node at (5,2) {$z_{i+1}$};
  \node at (7,0) {$z_i$};
  \node at (7,2) {$z_{i+1}$};
  \node at (-2,0){$+$};
  \node at (-2,2){$+$};
  \node at (8,0){$-$};
  \node at (8,2){$-$};
  \node at (8.5,0){$0$};
  \node at (8.5,2){$0$};
  \node at (-2.7,0){$c_{i+1}$};
  \node at (-2.7,2){$c_i$};
  \end{tikzpicture}}
  \]
In this case there are two possibilities for
the charges on the edges to the right of the
$R$-vertex (unless $c_i\equiv c_{i+1}$ modulo $n$)
and the two terms may be found again in
Figure~\ref{gamgamice}.

It must be checked that the statement remains true if $c_i=c_{i+1}$ though in
this case the two terms on the right-hand side in (\ref{partfe}) may be
combined. Since $g(0)=-v$ the coefficient is
\[\frac{(1-v)\mathbf{z}^{n\alpha_i}}{1-v\mathbf{z}^{n\alpha_i}} +
  \frac{-v(1-\mathbf{z}^{n\alpha_i})}{1-v\mathbf{z}^{n\alpha_i}} =
  \frac{-v+\mathbf{z}^{n\alpha_i}}{1-v\mathbf{z}^{n\alpha_i}}.
 \]
This is what we want by Figure~\ref{gamgamice}.
\end{proof}

Consider the following version of equation \eqref{partfe}:
  \begin{equation}
    \label{modpartfe}
  \widetilde{Z}(\mathfrak{S}_{\lambda,s_i\mathbf{z}};s_i\mathbf{c})=
  \widetilde{\operatorname{wt}} \left( \begin{array}{c} \gamgam{+}{+}{+}{+}{c_{i}}{c_{i+1}}{c_{i}}{c_{i+1}} \end{array} \right) 
  \widetilde{Z}(\mathfrak{S}_{\lambda,\mathbf{z}};s_i\mathbf{c})
  +\widetilde{\operatorname{wt}} \left( \begin{array}{c} \gamgam{+}{+}{+}{+}{c_{i}}{c_{i+1}}{\;\;c_{i+1}}{c_{i}} \end{array} \right)
  \widetilde{Z}(\mathfrak{S}_{\lambda,\mathbf{z}};\mathbf{c}),
  \end{equation}
where now $\widetilde{Z}$ is the partition function of the same ice model as before, but using the modified weights from Figure \ref{modifiedweights}, and $\widetilde{\operatorname{wt}}$ are also the modified weights from Figure \ref{modifiedrwt}. 

This may be compared with a result from Section~\ref{connex}.
Let $\varpi^\nu$ be a representative in $\tilde{T} / H$ with $\nu - \rho =
(c_1,\cdots,c_r)$ with $c_i \in [0,n)$.  Then for a simple reflection $s_i$
the following functional equation holds as explained in Proposition~\ref{KPlemma}:
\begin{equation} W_{\nu}^{{}^{s_i} \chi} \circ
  \overline{\mathcal{A}}_{s_i} (\pi(\varpi^\lambda) \phi_K) = \tau_{\nu,\nu}^1
  W_{\nu}^{\chi} (\pi(\varpi^\lambda) \phi_K) + \tau_{\nu, s_i \cdot
    \nu}^2 W_{{s_i \nu}}^\chi (\pi(\varpi^\lambda) \phi_K)
  \label{fundscat}.
\end{equation}

\begin{remark}
Combining the results of Proposition~\ref{taumatch} and Theorem~\ref{whitpartcharge}, we immediately
conclude that the right hand sides of (\ref{modpartfe}) and of
(\ref{fundscat}) are equal. We therefore obtain an interpretation of the
action of the intertwining operator on the Whittaker function at the ice model
level. To be more precise, the effect of the intertwining operator
$\overline{\mathcal{A}}_{s_1}$ on spherical Whittaker functions is realized by
swapping the roles of the parameters $z_i$ and $z_{i+1}$ in the ice model
while attaching an $R$-matrix at the edge of the system.
The effect of this operation is described by the Boltzmann weights
of the attached $R$-matrix.
\end{remark}

This shows that the functional equation of the Whittaker function has an
interpretation as equality of partition functions.
Such equivalences are useful when they transform a hard
problem in one area to an easy problem in a different area. Another example of
such an equivalence is given in~\cite{BBBG}. There, we
(together with Gray) prove the equality of the partition functions of two ice
models which is equivalent to the equality of two expressions for coefficients
of Weyl group multiple Dirichlet series. The proof of this fact occupies most
of \cite{wmd5book}; the proof is long and intricate. The proof in~\cite{BBBG}
on the other hand is very short and clear.

\bibliographystyle{hplain}
\bibliography{rmatrix}

\end{document}